\documentclass[11pt]{amsart}
\usepackage{srcltx}

\usepackage{amssymb}
\usepackage{amsmath}
\usepackage{amsthm}
\usepackage{amscd}
\usepackage{graphicx}
\usepackage{amsfonts}
\newtheorem{thm}{Theorem}[section]
\newtheorem{cor}[thm]{Corollary}
\newtheorem{lem}[thm]{Lemma}

\newtheorem{prop}[thm]{Proposition}

\theoremstyle{definition}
\newtheorem{defn}[thm]{Definition}
\theoremstyle{remark}
\newtheorem{rem}[thm]{Remark}
\newtheorem{ex}[thm]{Example}

\newcommand{\Hl }{\{ H_1, \ldots, H_m\} }
\renewcommand{\d }{{\rm dist}}

\newcommand{\dx }{\d_X}

\newcommand{\Lab}{\mathrm{Lab}}
\newcommand{\Con}{{\mathrm{Con}}}
\newcommand{\e }{\varepsilon }
\newcommand{\CG }{{\rm Con} ^{\omega } (G, d)}
\newcommand{\lio }{\mathrm{lim}^\omega }

\newcommand{\oas}{$\omega $--almost surely}

\newcommand{\T }{\Theta}
\newcommand{\lqt}{$\mathcal Q$-tree }
\newcommand{\lqst}{$\mathcal Q$-subtree\ }
\newcommand{\pp}{{\mathcal P}}
\newcommand{\qq}{{\mathcal Q}}

\newcommand{\la}{\langle}
\newcommand{\ra}{\rangle}
\newcommand{\rrr}{{\mathfrak r}}

\newcommand{\cgx}{\mathrm{Cay} (G,X)}

\newcommand{\G }{\mathrm{Cay} (G,X\cup\mathcal H)}

\newcommand{\lm}{{\lim}}
\newcommand{\dist}{{\mathrm{dist}}}

\newcommand{\Cay}{\mathrm{Cay}}

\newcommand {\Z}{\mathbb{Z}}            %% integers
\newcommand {\R}{\mathbb{R}} %% reals
\newcommand {\free}{\mathbb{S}} %% Space
\newcommand{\freg}{\mathbb{G}} %%freeg

\newcommand {\q}{\mathfrak q} %% quasi-geodesic
\newcommand {\g}{\mathfrak g} %% geodesic
\newcommand{\p}{\mathfrak p} %%geodesic

\newcommand {\pgot}{\mathfrak p}

\newcommand {\iv}{^{-1}}

\newcommand{\ttt}{{\mathcal T}}

\newcommand{\di}{\mathrm{div}}

\begin{document}

\title[Universal tree-graded spaces]{Universal tree-graded spaces and asymptotic cones}
\author{Denis Osin}\thanks{The research of the first author was supported in part by
NSF grant  DMS-1006345.}
\author{Mark Sapir}\thanks{The research of the second author was supported in part by NSF grant DMS-0700811.}
\address{Department of Mathematics,
Vanderbilt University, Nashville, TN 37240, U.S.A.}
\email{denis.v.osin@vanderbilt.edu}

\email{m.sapir@vanderbilt.edu} 

\maketitle
%\tableofcontents

\begin{abstract} We define and give explicit construction of a universal tree-graded space with a given collection of pieces. We apply that to proving uniqueness of asymptotic cones of relatively hyperbolic groups whose peripheral subgroups have unique asymptotic cones. Modulo the Continuum Hypothesis, we show that if an asymptotic cone of a geodesic metric space is homogeneous and has cut points, then it is the universal tree-graded space whose pieces are maximal connected subsets without cut points. Thus it is completely determined by its collection of pieces.
\end{abstract}

%%%%%%%%%%%%%%%%%%%%%%%%%%%%%%%%%%%%%%%%%%%%%%%%%%%%%%%%%%%%%%%%%%%%%%%%%%%%%%%%%%%%%%%%%%%%%%%%%%%

\section{Introduction}

%%%%%%%%%%%%%%%%%%%%%%%%%%%%%%%%%%%%%%%%%%%%%%%%%%%%%%%%%%%%%%%%%%%%%%%%%%%%%%%%%%%%%%%%%%%%%%%%%%%

Tree-graded spaces were introduced by Dru\c tu and the second author
in \cite{DS}. These are geodesic metric spaces equipped with a
collection of connected subsets, called pieces, such that every two
pieces intersect by at most one point and every simple closed loop
is contained in one piece. It turned out that tree-graded spaces
occur very often as asymptotic cones of groups. In fact every
geodesic metric space with cut points  has a natural tree-graded
structure: the pieces are maximal connected subspaces which have no
their own cut points \cite[Lemma 2.31]{DS}. (Recall that a point $s$
of a geodesic metric space $\free $ is a {\it cut point} if
$\free\setminus \{ s\} $ is not path connected).

It is proved in \cite[Section 3]{DMS} that a group has cut points in
its asymptotic cones if and only if the divergence function of this
group is super-linear. Asymptotic cones of relatively hyperbolic
groups \cite{DS}, mapping class groups \cite{Beh}, many right angled
Artin groups \cite{BC} have natural tree-graded structures. In
addition, it is shown in \cite[Section 7]{DS} that any
``sufficiently nice" metric space (e.g., compact and locally
contractible) appears as a piece in the tree-graded structure of
some asymptotic cone of a finitely generated group. In \cite[Section
4]{OOS}, it is proved that asymptotic cones of lacunary hyperbolic
groups given by presentations satisfying certain forms of small
cancelation have tree-graded asymptotic cones.

Existence of a tree-graded structure in asymptotic cones of a group $G$ can be used to prove some rigidity results. The approach is based of a similarity between the theory of groups acting on tree-graded spaces and that of groups acting on $\mathbb R$-trees. In particular, it helps answering questions about the number of pairwise non-conjugate homomorphisms from a group $\Gamma$ into $G$ \cite{DS1}, the size of $\mathrm{Out}(G)$, and whether $G$ is Hopfian or co-Hopfian \cite{DS1}. In \cite{BDS, BDS1}, the tree-graded structure of asymptotic cones of mapping class groups helped proving that a group with property (T) (and even much weaker properties) has only finitely many pairwise non-conjugate homomorphisms into a given mapping class group.

In this paper, we use the tree-graded structure of asymptotic cones to address the question about uniqueness of asymptotic cones of a group. It is now well known \cite{TV} that a finitely generated group may have several non-homeomorphic asymptotic cones. The same is true for finitely presented groups \cite{OlshSap}. In \cite{KSTT}, it is shown that provided the Continuum Hypothesis (CH) is true, the total number of non-isometric asymptotic cones of an arbitrary finitely generated group is at most continuum, while if CH is not true, there exists a finitely presented group with $2^{2^{\aleph_0}}$ pairwise non-homeomorphic asymptotic cones. There exists a finitely generated group with continuum pairwise non-$\pi_1$-equivalent asymptotic cones \cite[Theorem 7.37]{DS} regardless whether CH is true or not. On the other hand, all asymptotic cones of a finitely generated nilpotent group are bilipschitz equivalent and even isometric if the set of generators is fixed \cite{Pansu, Br}, and the asymptotic cones of all
non-elementary hyperbolic groups are isometric to the universal $\R$-tree, that is the complete $\R$-tree with branching number continuum at every point \cite{MNO}. An explicit construction of that tree is presented in \cite{EP}.

In order to present the main results of the paper, we need the following definition.

\begin{defn} \label{def1} Let $\qq$ be a collection of geodesic metric spaces. We say that a geodesic metric space $\free$ is a
{\it $\qq $-tree} if $\free$ is tree-graded \cite[Section 1]{DS} with respect to pieces isometric to  elements of $\qq$. If elements of $\qq$ are without cut points, then $\qq$-trees are precisely the geodesic metric spaces $\free$ where every maximal
connected subset of $\free$ without cut points is isometric to some $Q\in \qq$. \footnote{One can also define a category of $\qq$-trees in a natural way, including the collection of pieces as a part of the structure of objects and defining morphisms as maps preserving the structure, but that is not used in this paper.}
\end{defn}

Let us point out that all metric spaces considered in this paper are of sizes at most continuum and we shall sometimes omit this condition.

We say that a $\qq$-tree is {universal} if for every point $s\in
\free$, the cardinality of the set of connected components of
$\free\setminus \{s\}$ of any given {\em type} (for the definition
of type see Section \ref{5}) is continuum. This notion generalizes
the notion of universal $\mathbb R$-trees studied by Mayer, Nikiel,
and Oversteegen  \cite{MNO} as well as Erschler and Polterovich
\cite{EP}, where pieces are points and all connected components of
$\free\setminus \{s\}$ are of the same type. A discrete version of
$\qq$-trees was also studied by Quenell \cite{Q}.

\begin{thm} \label{main1} Let $\mathcal Q$ be a collection of homogeneous complete geodesic metric spaces. Then the following hold.
\begin{enumerate}
\item There exists a universal $\qq $-tree.

\item Every $\qq $-tree of cardinality at most continuum embeds into a universal $\qq$-tree.

\item Every two universal $\mathcal Q$-trees are isometric.
\end{enumerate}
\end{thm}

In fact, the isometry in part (3) of the theorem preserves the tree
graded structure (i.e., maps pieces to pieces). A priori this is not
obvious since tree graded structures on metric spaces are, in
general, not unique.

We present an explicit construction of the universal $\qq$-tree,
using the notion of the {\em tree product of metric spaces} similar
in spirit to the construction in \cite{EP}.

\begin{thm}\label{main2} \label{rhg-intr}
Let $G$ be a finitely generated relatively hyperbolic group. Then for every non-principal ultrafilter $\omega $ and every scaling sequence $d=(d_i)$, the asymptotic cone $\CG $ is bi-Lipschitz equivalent to the universal $\mathcal Q$-tree, where $\mathcal Q$ consists of asymptotic cones of peripheral subgroups with respect to $\omega $ and $d$.
\end{thm}

The following two corollaries follow from Theorem \ref{rhg-intr} and
Lemma \ref{ble}.

\begin{cor}\label{cor1}
Let $G_1$ and $G_2$ be two finitely generated relatively hyperbolic groups with the same set of isomorphism classes of peripheral subgroups. Then for any fixed ultrafilter $\omega $ and scaling sequence $d$, the asymptotic cones $\Con ^\omega (G_1, d)$ is bilipschitz equivalent to $\Con ^\omega (G_1, d)$. In particular, the sets of asymptotic cones of $G_1$ and $G_2$ coincide.
\end{cor}

\begin{ex}
Let $G$ be the fundamental group of a hyperbolic knot complement. Then it is hyperbolic relative to a free abelian subgroup of rank $2$ \cite{Gro,F} and all asymptotic cones of $G$ are bi-Lipschitz equivalent to the universal $\{{\mathbb R}^2\}$-tree. The same holds, say, for asymptotic cones of $(\mathbb Z\times \mathbb Z)\ast \mathbb Z$. Similarly, every non-uniform lattice in $\mathrm{SO}(n,1)$ is relatively hyperbolic with respect to finitely generated free Abelian subgroups $\Z^{n-1}$, hence their asymptotic cones are all bi-Lipschitz equivalent to the asymptotic cones of $\Z^{n-1}\ast \Z$ and are bi-Lipschitz equivalent to the universal $\{\R^{n-1}\}$-tree. More generally, every non-uniform lattice $\Gamma$ in a rank 1 semi-simple Lie group  is relatively hyperbolic with respect to nilpotent subgroups \cite{Gro,F}. The asymptotic cones of a nilpotent group are homeomorphic to $\R^k$ for some $k$ \cite{Pansu}. Hence every asymptotic cone of $\Gamma$ is homeomorphic to the universal $\{\R^{k}\}$-tree, where $k+1$ is the dimension of the associated rank one symmetric space.
\end{ex}

We say that a finitely generated group $G$ is {\it homeo-unicone}
(respectively {\em bi-Lipschitz-unicone}) if all its asymptotic
cones are homeomorphic (respectively bi-Lipschitz equivalent). For
example, all nilpotent groups and hyperbolic groups are
bi-Lipschitz-unicone. The fact that some non-nilpotent solvable
groups (for example, SOL) are bi-Lipschitz-unicone is proved in
\cite[Section 9]{deC}.

\begin{cor}\label{cor2}
If a finitely generated group $G$ is hyperbolic relative to a collection of homeo-unicone (respectively, bi-Lipschitz-unicone) subgroups, then $G$ is homeo-unicone (respectively, bi-Lipschitz-unicone).
\end{cor}

\begin{thm}\label{main3} [Assuming CH is true] Let $\free $ be an asymptotic cone of a geodesic metric space. Suppose that $\free $ is homogeneous and has cut points. Then $\free $ is isometric to the universal $\qq $-tree, where $\qq$ consists of representatives of isometry classes of maximal connected subspaces of $\free$ without cut points.
\end{thm}

Thus, in particular, modulo CH, the asymptotic cones of the mapping class groups are completely determined by their pieces, that is maximal connected subsets without cut points. These have been described in \cite{BKMM} and \cite{BDS}. But we still do not know if the pieces depend on the sequence of scaling constants or the ultrafilter. Hence the question of whether every mapping class group of a punctured surface has unique up to homeomorphism asymptotic cone (see \cite[Question 7.7]{Beh}) remains open.

The paper is organized as follows. In Section 2 we collect some necessary definitions and results about tree graded spaces and relatively hyperbolic groups. Sections 3 and 4 contain the definition,  explicit construction, and main geometric properties of tree products. In Section 5 we introduce universal $\qq$-trees and prove Theorem \ref{main1}. Theorems \ref{main2} and \ref{main3} are proved in Sections 6 and 7, respectively.

{\bf Acknowledgement.} The authors are grateful to Cornelia Dru\c tu for helpful discussions. In particular, Sections \ref{sec2}, \ref{sec3} were written in collaboration with her in 2004 during the preparation of \cite{DS} but were not included in \cite{DS} because we were not able to prove the uniqueness of universal $\qq$-trees then.  We are also grateful to Yves de Cornulier for his comments.
%%%%%%%%%%%%%%%%%%%%%%%%%%%%%%%%%%%%%%%%%%%%%%%%%%%%%%%%%%%%%%%%%%%%%%%%%%%%%%%%%%%%%%%%%%%%%%%%%%%

When this paper was completed, it was brought to our attention by Cornelia Dru\c tu that a preprint by Alessandro Sisto \cite{Si2} that appeared at the same time contains results similar to the results proved in  the present paper, including Corollary \ref{cor1}. The preprint \cite{Si2} is based on the thesis \cite{Si1}.

\section{Preliminaries}

%%%%%%%%%%%%%%%%%%%%%%%%%%%%%%%%%%%%%%%%%%%%%%%%%%%%%%%%%%%%%%%%%%%%%%%%%%%%%%%%%%%%%%%%%%%%%%%%%%%

\noindent{\bf Notation.} All generating sets of groups are supposed
to be symmetric, i.e., to be closed under taking inverse elements.
Given a group $G$ generated by a subset $X\subseteq G$, we denote by
$\cgx $ the Cayley graph of $G$ with respect to $X$ and by $|g|_X$
the word length of an element $g\in G$. If $p$ is a (combinatorial)
path in $\cgx$, $\phi (p)$ denotes its label, $\Lab(p)$ denotes its
length, $p_-$ and $p_+$ denote its starting and ending vertex. The
notation $p^{-1}$ will be used for the path in $\cgx$ obtained by
traversing $p$ backwards. By saying that $o=p_1\dots p_k$ is a cycle
in $\cgx$ we will mean that $o$ is obtained as a consecutive
concatenation of paths $p_1,\dots p_k$ such that
$(p_{i+1})_-=(p_i)_+$ for $i=1,\dots,k-1$ and $(p_k)_+=(p_1)_-$. For
a word $W$ written in the alphabet $X^{\pm 1}$, $\|W\|$ will denote
its length. For two words $U$ and $V$ we shall write $U \equiv V$ to
denote the letter-by-letter equality between them.

\medskip

\noindent{\bf Tree-graded spaces.}
We collect here all the necessary definitions and basic properties
of tree-graded spaces from \cite{DS} needed in this paper.

Recall that a metric space $\mathbb X$ is called {\it geodesic} (or a {\it length space}) if every two points of $\mathbb X$ can be joined by a geodesic. A point $x\in \mathbb X$ is a {\it cut point} if $\mathbb X\setminus \{ x\} $ is disconnected. Note that for a geodesic space, connected components of $\mathbb X\setminus \{ x\} $  are path components.

\begin{def}\label{tgspace}
Let $\free$ be a complete geodesic metric space and let $\pp$ be a
collection of closed geodesic non-empty subsets (called
{\it{pieces}}). Suppose that the following two properties are
satisfied:

\begin{enumerate}

\item[($T_1$)] Every two different pieces have at most one common
point.

\item[($T_2$)] Every non--trivial simple geodesic triangle (a
simple loop composed of three geodesics) in $\free$ is contained
in one piece.
\end{enumerate}

Then we say that the space $\free$ is {\em tree-graded with
respect to }$\pp$.
\end{def}

\noindent  We allow $\mathcal
P$ to be empty. Clearly $\free$ is tree--graded with respect to
the empty collections of pieces only if $\free$ is a tree.
%In this paper we also assume that no piece is a single point.

Recall that a (topological) {\it arc} in a metric space $S$ is any continuous injective map from a segment of $\mathbb R$ to $S$. The following was proved in \cite[Corollary 2.10, Proposition 2.17]{DS}.

\begin{lem}\label{DS}
For any tree-graded space $\free$ the following hold.

\begin{enumerate}
\item[(a)] Every
non-empty intersection of a topological arc in $\free$ and a
piece is a point or a sub-arc.

\item[(b)] Every simple loop in $\free$ is contained
in one piece.

\item[(c)] For every arc $c:[0,d]\to T$, where $c(0)\ne c(d)$, and any $t\in
[0,d]$, let $c[t-a,t+b]$ be a maximal sub-arc of $c$
containing $c(t)$ and contained in one piece. Then every other
topological arc with the same endpoints as $c$ must contain the
points $c(t-a)$ and $c(t+b)$.
\end{enumerate}
\end{lem}

For every point $x$ in a tree-graded space $\free$ let $T_x$ be the set of points $y\in \free$ such that every geodesic $[x,y]$ intersects every piece in at most one point. By \cite[Lemma 2.14]{DS}, $T_x$ is an $\mathbb R$-tree, it is called the {\em transversal tree of $\free$ at $x$}.

\begin{lem}\label{txy}(\cite[Lemma 2.13]{DS})
Let $x\in \free $ and $y\in T_x$. Then $T_x=T_y$.
\end{lem}

\begin{lem}\label{ctree}(\cite[Lemmas 2.19, 2.14]{DS}
For every $x\in \free$, $T_x$ is a closed subset of $\free$. Every simple path connecting two points in $T_x$ inside $\free$ is contained in $T_x$.
\end{lem}

\begin{lem} \label{treesrem}(\cite[Remark 2.27]{DS}) Let $\free$ be a metric space that is tree-graded with respect to a collection of pieces $\pp$. Let
$\{T_i\mid i\in I\}$ be the collection of all transversal trees $T_x$ in $\free$. Then
the set $\pp'=\pp\cup \{ T_i \mid i\in I \}$ also satisfies
properties $(T_1)$ and $(T_2)$, in particular $\free$ is tree-graded with respect to the collection of pieces $\pp'$. In that tree-graded space all transversal trees are trivial.
\end{lem}

\begin{lem}[{\bf Triangles in tree-graded spaces}] \label{ttg}
Let $\Delta =pqr$ be a geodesic triangle in a tree-graded space $\free$. Then the sides $p, q, r$ can be decomposed as $p=p_1p_2p_3$, $q=q_1q_2q_3$, and $r=r_1r_2r_3$ (some of $p_i,q_i,r_i$ may be trivial), where $(p_1)_+=(r_3)_-$, $(p_3)_-=(q_1)_+$, $(q_3)_-=(r_1)_+$, and the cycle $p_2q_2r_2$ either is a point or belongs to a single piece.
\end{lem}

\proof Let $A$ be the last common point of $p$ and $r\iv$, $B$ be the last common point of $p\iv$ and $q$, $C$ be the last common point of $q\iv$ and $r$. Then either $A=B=C$ or the union of  $[A,B]\subseteq p$,  $[B,C]\subseteq  q$ and  $[C,A] \subseteq r$ is a simple geodesic triangle in $\free$ and by the definition of tree-graded spaces it is contained in one piece of $\free$. Then the points $A,B$ (resp. $B,C$ and $C,A$) gives the required decomposition $p = p_1,p_2, p_3$ (resp. $q=q_1q_2q_3$, $r=r_1r_2r_3$).
\endproof

\medskip

\noindent{\bf Asymptotic cones.}
Recall that a non-principal ultrafilter $\omega$ is a finitely additive measure
defined on all subsets $S$ of ${\mathbb N}$, such that $\omega(S)
\in \{0,1\}$, $\omega (\mathbb N)=1$, and $\omega(S)=0$ if $S$ is
a finite subset. For a bounded sequence of numbers $x_n$, $n\in
\mathbb N$, the limit $\lio x_n$ with respect to $\omega$ is the unique
real number $a$ such that $\omega(\{i\in {\mathbb N}:
|x_i-a|<\epsilon\})=1$ for every $\epsilon>0$.

Let $(X,\dist)$ be a metric space. Fix an arbitrary
sequence $e=(e_n)$ of points $e_n\in X$ and a
{\it scaling sequence} of positive real numbers $d=(d_n)$ with $\lim^\omega(d_n)=\infty$ (that is for every $K>0$ the set $\{n: d_n>K\}$ is in $\omega$). Let $X_i=\left(X, \frac1{d_i}\dist \right)$ be the scaled copy of $X_i$. Consider the subset $\mathcal B$ of the ultraproduct $\Pi ^\omega  X_i$ consisting of all $x=(x_n)^\omega \in \prod ^\omega  X_i$ such that
$\lio \dist (x_n,e_n)/d_n  < \infty $. Two sequences
$(x_n)^\omega$ and $(y_n)^\omega $ from this set ${\mathcal B}$ are said to be {\em
equivalent} if $\lio\dist(x_n,y_n) =0$. The {\em asymptotic cone} $\Con^\omega(X,e, d)$ is the quotient of $\mathcal B$ modulo
this equivalence relation. The metric is defined by
$$
\dist ([(x_n)^\omega ], [(y_n)^\omega] )=\lio (\dist (x_n, y_n)/d_n) .
$$
It is easy to verify that $\dist $ is well-defined (i.e., it is
independent of the choice of representatives of the equivalence
classes $[(x_n)^\omega ]$ and $[(y_n)^\omega ]$) and indeed
satisfies all axioms of a metric.

Note that if $e$ is a constant sequence, the asymptotic cone is
independent of $e$ and we write $\Con^\omega(X,d)$ instead of
$\Con^\omega(X,e,d)$. If $X$ is homogeneous and $e$ is an arbitrary
sequence, then $\Con^\omega(X, e,d)$ is isometric to
$\Con^\omega(X,d)$.

An {\em asymptotic cone} of a finitely generated group $G$ with a
word metric is the asymptotic cone of its Cayley graph (considered
as the discrete space of vertices with the word metric). Asymptotic
cones corresponding to two different finite generating sets of $G$
(and the same ultrafilters and scaling constants) are bi-Lipschitz
equivalent. The asymptotic cone $\Con^\omega(G,d)$ of a group $G$ is
a homogeneous geodesic metric space.

\medskip

\noindent{\bf Relatively hyperbolic groups.}
Let us recall one of the definitions of finitely generated relatively hyperbolic groups (see \cite[Appendix]{DS}).

\begin{defn} \label{rhd} A group $G$ generated by a finite set $X$ is called {\it hyperbolic relative to a collection of proper
subgroups} $\Hl $ if for every non-principal ultrafilter $\omega $ and every scaling sequence $d=(d_n)$ the following conditions hold.
\begin{enumerate}
\item[(a)]
Let $\mathcal L$ be the collection of all limits $\lio{g_jH_i}$
in $\CG $, where $|g_j|=O(d_j)$. Then two limits $\lio
g_jH_i,\lio g_j^\prime H_k\in \mathcal L$ coincide in $\CG $ if
and only if $i=k$ and $g_jH_i=g_j^\prime H_i$ \oas.

\item[(b)] The asymptotic cone $\Con^\omega(G,(d_n))$ is tree-graded with respect to the set of all (distinct) elements of $\mathcal L$.
\end{enumerate}
\end{defn}

An element $g\in G$ is called {\it parabolic} if it is conjugate to an element of one of the
subgroups $H_i$. An element is said to be {\it hyperbolic} if it is not
parabolic and has infinite order. Recall also that a group is {\it elementary} if it contains a cyclic
subgroup of finite index. The first part of the next lemma is well-known in the context of convergence groups \cite{Tuk,Yam}. For the second part we refer to \cite{Osi06b}.

\begin{lem}\label{Eg}
Let $g$ be a hyperbolic element of $G$. Then the following conditions hold.
\begin{enumerate}
\item[(a)] There is a unique maximal elementary
subgroup $E_G(g)\le G$ containing $g$.

\item[(b)] The group $G$ is hyperbolic relative to the collection
$\{H_1, \ldots , H_m, E_G(g)\} $.
\end{enumerate}
\end{lem}

Let $$\mathcal H=\bigsqcup\limits_{i=1}^m H_i\setminus\{ 1\} .$$ Let
$q$ be a path in the Cayley graph $\G $. A (non--trivial) subpath
$p$ of $q$ is called an {\it $H_\lambda $--subpath} for some
$\lambda \in \{1,\ldots,m\}$, if the label of $p$ is a word in the
alphabet $H_\lambda\setminus \{ 1\} $. If $p$ is a maximal
$H_\lambda $--subpath of $q$, i.e. it is not contained in a bigger
$H_\lambda $--subpath, then $p$ is called an {\it $H_\lambda
$--component} (or simply a {\it component}) of $q$.

Two $H_\lambda $--subpaths (or $H_\lambda $--components) $p_1, p_2$
of a path $q$ in $\G $ are called {\it connected} if there exists a
path $c$ in $\G $ that connects some vertex of $p_1$ to some vertex
of $p_2$ and ${\phi (c)}$ is a word consisting of letters from $
H_\lambda\setminus\{ 1\} $. In algebraic terms this means that all
vertices of $p_1$ and $p_2$ belong to the same coset $gH_\lambda $
for a certain $g\in G$. Note that we can always assume that $c$ has
length at most $1$, as every nontrivial element of $H_\lambda
\setminus\{ 1\} $ is included in the set of generators.  An
$H_\lambda $--component $p$ of a path $q$ is called {\it isolated }
(in $q$) if no distinct $H_\lambda $--component of $q$ is connected
to $p$.

The following lemma is a particular case of \cite[Lemma 2.27]{Osi06a} applied to finitely generated relatively hyperbolic groups. Given a path $p$ in $\G$, we define its {\it $X$-length } by
$$
l_X (p)=\dx (p_-, p_+).
$$

\begin{lem}\label{Omega}
There exists a constant $L>0$ such that the following condition
holds. Let $q$ be a cycle in $\G $, $p_1, \ldots , p_k$ a set of
isolated components of $q$. Then we have
$$ Ll(q)\ge \sum\limits_{i=1}^k l_X (p_i).$$
\end{lem}

%%%%%%%%%%%%%%%%%%%%%%%%%%%%%%%%%%%%%%%%%%%%%%%%%%%%%%%%%%%%%%%%%%%%%%%%%%%%%%%%%%%%%%%%%%%%%%%%%%%%%%%%%

\section{Tree products of metric spaces}\label{sec2}

%%%%%%%%%%%%%%%%%%%%%%%%%%%%%%%%%%%%%%%%%%%%%%%%%%%%%%%%%%%%%%%%%%%%%%%%%%%%%%%%%%%%%%%%%%%%%%%%%%%%%%%%%

Let $\pp=\{M_i\mid i\in I\}$ be a set of complete geodesic metric
spaces, and let $\ttt$ be the complete $\R$-tree of degree continuum \cite{MNO} explicitly constructed in \cite{EP}. In this section, we define a universal
tree-graded space with every piece isometric to a spaces from
$\pp$ and transversal trees isometric to $\ttt$. We shall call it the {\em tree product} of $\pp$ and denote
it by $\Pi\pp$.

Our construction of the tree product is similar in spirit to the construction from \cite{EP}. The idea is the following. For every point $y$ of a tree-graded
space $\free$ we consider a geodesic connecting that point with a
fixed base point $x$. Unlike in the case of trees in
\cite{EP}, there could be many geodesics connecting $x$ and $y$ in
$\free$. But by Lemma \ref{DS}, every such geodesic must pass
through a certain sequence of points depending only on $x, y$. These
are the points where the geodesic enters and exits pieces or the transversal trees.
Following a geodesic from $x$ to $y$, we can remember only the
entrance and exit points. Thus to the set of all geodesics
$c\colon [0,d]\to \free$ between two points $x,y$, we will
associate one locally constant function defined on a dense open
subset of $[0,d]$ whose values are pairs of distinct points from
pieces. The function is not defined in the ``break points'' where
the geodesic exits one piece (or transversal tree) and enters another. Given $x$, that
function uniquely describes $y$.

Obviously there are several necessary restrictions on the
functions we get stemming from the fact that geodesics cannot
backtrack, and that two ``consecutive'' values of the function do
not belong to the same piece (tree): there are no ``fake'' exits.

Here is the formal definition. We divide it into two
parts. First we define the set $\Pi\pp$ and then the metric.

\begin{defn}[tree product: the set] \label{set} For simplicity we shall denote the universal $\R$-tree of degree continuum $\ttt$ from \cite{MNO,EP} as $M_0$.\footnote{The definition is valid without this assumption. But if, say, we do not include a tree in the list of pieces at all, the resulting tree-graded space will have trivial (single point) transversal trees as follows from the description of geodesics in Lemma \ref{every} below.}
For every $\alpha\in I\cup\{0\}$ let $\Omega_\alpha$ be the set of
all pairs of distinct points of $M_\alpha$, i.e. $M_\alpha\times
M_\alpha \setminus \{(x,x) \mid x\in M_\alpha\}$. Let
$\Omega=\bigsqcup_{\alpha\in I\cup\{0\}}\Omega_\alpha$.

Let $\Pi\pp$ be the set of partial functions $f\colon \R\to
\Omega$ that consists of
all functions $f$ with the following properties:
\begin{itemize}
\item[($\Pi_1$)] The domain of $f$ is a dense open subset $A_f$ of
an open interval $]0,d(f)[$, $0\le d(f)<\infty$.\footnote{We
denote an open interval with endpoints $a,b$ by $]a,b[$,
half-open intervals by $]a,b], [a,b[$, and closed intervals by
$[a,b]$.} The only $f\in \Pi\pp$ with $d(f)=0$ is the empty
function and is denoted by $f_\emptyset$.

\item[($\Pi_2$)] $f$ is locally constant on $A_f$.

\item[($\Pi_3$)] For every maximal sub-interval $u$ of $A_f$ if
$f(u)= \{(x,y)\}$ then the length $|u|$ equals the distance
$\dist(x,y)$ between $x$ and $y$ in the corresponding $M_\alpha$,
$\alpha\in I\cup\{0\}$.

\item[($\Pi_4$)] ({\bf No fake exits.}) No two consecutive values
of $f$ have the form $(x_1,x_2), (x_2,x_3)$.

\item[($\Pi_5$)] ({\bf No backtracking.}) Let
$\overline{(x,y)}=(y,x)$ for every $(x,y)\in \Omega$. For every
$x\in \R$ let $\sigma_x$ be the symmetry in $\R$ with respect to
$x$. Then for every non-empty open sub-interval $]p-q,p+q[$ of
$]0,d(f)[$, the (partial) functions $\overline{f \circ
\sigma_p}$ and $f$ do not coincide on $]p-q,p+q[$.
\end{itemize}
\end{defn}

\begin{rem} \label{backward} The informal meaning of ($\Pi_4$)
is that the maximal sub-intervals of $A_f$ divide a geodesic into
maximal sub-geodesics each belonging to one piece (every exit
point should lead to a different piece). Note also that $A_f$ can have infinitely many connected components.

A function $f\colon ]0,d[\to \Omega$ as above will also be responsible for representing geodesics connecting the point $f$ to the point $f_{\emptyset}$. Similarly $f|_{]p,q[}$ corresponds to a geodesic $c_1$ connecting $f(p) $ to $f(q)$, and $\overline{f \circ \sigma_p}$ corresponds to a geodesic $c_2$ connecting $f(p)$ to $f(2p-q)$, which is symmetric to $c_1$ with respect to $f(p)$. Informally, the property ($\Pi_5$) means that $c_1$ and $c_2$ are never represented by the same function, which makes the correspondence between functions and geodesics well-defined.

Note also that if $I$ and all $M_\alpha $ have cardinality at most
continuum, then $\Pi M_\alpha$ has cardinality at most continuum as
well.
\end{rem}

In order to define a metric on $\Pi\pp$ we need some more notation
and a lemma.

\noindent \textit{Notation}: Let $f, g$ be two functions from
$\Pi\pp$. Let $\di(f,g)$ be the point of diverging, that is the
supremum of all numbers $s'$ such that the restrictions of $f$ and $g$ to
$]0,s'[\cap A_f$ coincide. We note that the set of $s'$ is non-empty as it
contains $0$.

\begin{lem} $\di(f,g)\not\in A_f\cup A_g$.
\end{lem}
\proof The statement is obviously true if $s=0$. If $s>0$ and if
for instance $s\in A_f$ then there exists a maximal interval
$]s-a,s+b[$ on which $f$ is constant equal to $(x,y)$ with
$\hbox{dist}(x,y)=b+a$. On $]s-a,s[$ the function $g$ is also
constant, equal to $(x,y)$. Now ($\Pi _3$) implies that $g$ is constant equal
to $(x,y)$ on $]s-a,s+b[$. This contradicts the maximality of
$s$.\endproof

\noindent \textit{Notation}: If $s=\di(f,g)$, let $]s,a_g(f)[$ and
$]s,a_f(g)[$ be the maximal (possibly empty) sub-intervals of
$A_f$ and $A_g$ respectively. If $a_g(f)\ne s$ and $a_f(g)\ne s$
then $f$ is constant on $]s,a_g(f)[$, $g$ is constant on $]s,a_f(g)[$.
In that case let us denote the values of $f$ and $g$ on these
(non-empty) intervals by $(x_g(f), y_g(f))$ and $(x_f(g), y_f(g))$
respectively.

Notice that if $x_f(g)=x_g(f)$ then, informally speaking, $a_g(f)\ne s$, $a_f(g)\ne s$ means
that after the two geodesics corresponding to $f, g$ diverge, they
still travel in the same piece or on the same transversal tree for some time,
but they exit this piece or tree through different points $y_g(f)$
and $y_f(g)$ respectively ($y_g(f)$ and $y_f(g)$ are necessarily
different because otherwise $s$ would not be the divergence
point).

Now we are ready to define a metric on $\Pi\pp$.

\begin{defn}[tree product: the metric]\label{metric}
Let $f,g\in \Pi\pp$. Define the number $D(f,g)$ as follows.

\noindent {\bf Case 1.} Suppose that $a_g(f)\ne \di(f,g)$,
$a_f(g)\ne \di(f,g)$, $x_f(g)=x_g(f)$. Then set

$$D(f,g)=d(f) -a_g(f) +d(g)
-a_f(g) + \dist(y_g(f),y_f(g)).
$$

\noindent{\bf Case 2.} If $a_g(f)=\di(f,g)$ or $a_f(g)=\di(f,g)$
or $x_f(g)\ne x_g(f)$ then we set
$$D(f,g)=d(f)-s+d(g)-s.$$
\end{defn}

\begin{rem} It is easy to see that the function $D$ that we
have just defined is the most natural candidate for a metric in
$\Pi\pp$. Informally speaking if $f, g$ correspond to points $u,
v$, and $c, c'$ are geodesics connecting $u, v$ to the base point
in a tree-graded space then in Case 1, every geodesic connecting
$u$ with $v$ must pass through the points $y_g(f)$ and $y_f(g)$,
in particular the curve
$$c\iv|_{]d(f),a_g(f)[} \cup [y_g(f),y_f(g)]\cup c'|_{]a_f(g),d(g)[}$$ should be a
geodesic between $u$ and $v$ whose length is $$D(f,g)=d(f) -a_g(f)
+d(g) -a_f(g) + \dist(y_f,y_g).
$$
In Case 2, every geodesic from $u$ to $v$ must pass through
$c(\di(f,g))$ and the curve $$c\iv|_{]d(f), \di(f,g)[}\cup
c'|_{]\di(f,g),d(g)[}$$ should be a geodesic connecting $u$ and
$v$ whose length is $$D(f,g)=d(f)-s+d(g)-s.$$

\end{rem}

We say that a pair of functions $(f,g)$ is \textit{of type} 1 if
they satisfy the conditions of Case 1. Otherwise we say that the
pair of functions $(f,g)$ is \textit{of type} 2.

\noindent \textit{Notation}: For every $f\in \Pi\pp$ and every
open interval $]a, b[\subset ]0,d(f)[$ we denote by $f|_{]a, b[}$ the
restriction of $f$ onto $]a, b[$. We denote $f|^\iota_{]a, b[}$
the partial function on $]0 ,b-a[ $ obtained from $f|_{]a, b[}$ by
pre-composition with the addition with $a$. It is obvious that
$f|^\iota_{]a,b[}$ belongs to $\Pi\pp$.

Let $f,g\in \Pi\pp$. We denote $f \sqcup g$ the function defined
on $]0, d(f)+d(g)[$ which coincides with $f$ on $]0,d(f)[$ and
with $g$ pre-composed with the subtraction of $d(f)$ on $(d(f),
d(f)+d(g))$. We note that it is not always true that $f\sqcup g$
is in $\Pi\pp$, as conditions ($\Pi_4$) or ($\Pi_5$) might not be
satisfied.

\begin{rem} Informally, $f\sqcup g$ is the function corresponding to the curve
obtained as a composition of the geodesic corresponding to $f$ and
a geodesic ``parallel'' to the geodesic corresponding to $g$. The
resulting curve may be not a geodesic because, for example, some
backtracking can occur at the point where the geodesics meet.
\end{rem}

\noindent\textit{Notation}: For every $(x,y)\in \Omega $ we denote
by $\mathbf{C}_{(x,y)}$ the function defined on $]0,\dist(x,y)[$
and constant equal to $(x,y)$. It is obvious that
$\mathbf{C}_{(x,y)}$ belongs to $\Pi\pp$.

\begin{lem}\label{D}
(1) For every $f,g\in \Pi\pp$ we have
$$
|d(f)-d(g)|\leq D(f,g)\leq d(f)-\di(f,g)+d(g)-\di(f,g).
$$

(2) For every $f,g_1,g_2 \in \Pi\pp$ such that $f\sqcup g_1,
f\sqcup g_2 \in \Pi\pp$ we have
$$
D(f \sqcup g_1,f \sqcup g_2 )=D(g_1,g_2)\, .
$$

(3) If $(f,g)$ is of type 1 then $D(f,g)>0$.
\end{lem}

\proof (1) Suppose that $(f,g)$ is of type 1. Let
$x=x_f(g)=x_g(f)$. We begin with the second inequality. We have
$a_g(f)=\di(f,g)+\dist(x,y_g(f))$ and
$a_f(g)=\di(f,g)+\dist(x,y_f(g))$. Hence
$D(f,g)=d(f)-\di(f,g)+d(g)-\di(f,g)
-(\dist(x,y_g(f))+d(x,y_f(g))-\dist(y_g(f),y_f(g)))\leq
d(f)-\di(f,g)+d(g)-\di(f,g)$ by the triangle inequality that
satisfied by $\dist$.

To prove the first inequality it suffices to prove that
$2(d(g)-\di(f,g))-[d(x,y_g(f))+d(x,y_f(g))-d(y_g(f),y_f(g))]\geq
0$. Since $\dist(x,y_f(g))\le d(g)-\di(f,g)$, it is enough to show
that $d(g)-\di(f,g)\geq d(x,y_g(f)-\dist(y_g(f),y_f(g))$. This
follows again from the triangle inequality since $d(g)-\di(f,g)\ge
\dist(x,y_f(g))$ and $\dist(x,y_f(g))\ge
d(x,y_g(f)-\dist(y_g(f),y_f(g))$.

Suppose now that $(f,g)$ is of type 2 then the second inequality
is obvious (it is in fact equality). The first inequality follows
from the obvious fact that $\di(f,g)\leq \min (d(f),d(g))$.

Statement (2) is an immediate consequence of the definition of
$D$.

(3) If $(f,g)$ is of type 1 then $D(f,g)=0$ implies that
$y_f=y_g$. The contradicts the maximality of $\di(f,g)$.\endproof

\begin{prop}\label{trin}
$D$ is a metric on $\Pi\pp$.
\end{prop}

\begin{proof}
(1)\quad Suppose that $D(f,g)=0$. Then $(f,g)$ is of type 2 and
$d(f)=d(g)=\di(f,g)$, so $f=g$ (since $f$ and $g$ coincide on
$(0,\di(f,g))$.
(2)\quad Let $f,g,h$ be arbitrary functions in $\Pi\pp$. Let us
check the triangle inequality:
$$
D(f,g)\leq D(f,h)+D(h,g).
$$

\noindent{\bf Case I.} Suppose $\di(f,h)=\di(h,g)=s$. It follows
that $\di(f,g)\geq s$.

(a) Suppose that $(f,h)$ and $(h,g)$ are both of type 1. Then
either $\di(f,g)=s$ and then $a_g(f)=a_h(f)$, $a_f(g)=a_h(g)$,
$y_g(f)=y_h(f)$, $y_g(h)=y_f(h)$ and $y_f(h)=y_g(h)$, or
$\di(f,g)>s$ and then $\di(f,g)>a_h(f)$, $\di(f,g)>a_h(g)$.

We have $D(f,h)+D(h,g)=
d(f)-a_h(f)+d(h)-a_g(h)+\dist(y_h(f),y_f(h))+d(h)-a_f(h)+d(g)-a_h(g)+\dist(y_h(g),y_g(h))\geq
d(f)-a_h(f) + d(g)-a_h(g) +
\dist(y_h(f),y_f(h))+\dist(y_h(g),y_g(h))$.

In the case $\di(f,g)=s$, $a_g(f)=a_h(f)$, $a_f(g)=a_h(g)$,
$y_g(f)=y_h(f)$, $y_f(h)=y_g(h)$ the last term is not smaller than
$D(f,g)=d(f)-a_g(f)+d(g)-a_f(g)+\dist(y_g(f),y_f(g))$ by the
triangle inequality for $\dist$.

In the case $\di(f,g)>s$, $\di(f,g)>a_h(f)$, $\di(f,g)>a_h(g)$  we
continue with $D(f,h)+D(h,g)\geq d(f)-\di(f,g) + d(g)-\di(f,g)$
and it remains to apply Lemma \ref{D}, (1).

(b) If $(f,h)$ is of type 1 and $(h,g)$ of type 2 then
$\di(f,g)=s$ and $(f,g)$ is of type 2. We have $$D(f,h)+D(h,g) =
d(f)-a_h(f)+d(h)-a_f(h)+\dist(y_h(f),y_f(h))+d(h)-s+d(g)-s.$$
Since $\dist(y_h(f),y_f(h))+(d(h)-s)\ge
\dist(y_h(f),y_f(h)+(a_f(h)-s)=\dist(y_h(f),y_f(h))+\dist(y_f(h),x_h(f))
\ge \dist(y_h(f),x_h(f))=a_h(f)-s$ by the triangle inequality, we
can deduce that $$D(f,h)+D(h,g)\ge
d(f)-a_h(f)+a_h(f)-s+d(g)-s=d(f)-s+d(g)-s=D(f,g).$$

(c) If $(f,h)$ and $(h,g)$ are both of type 2 then
$D(f,h)+D(h,g)=d(f)-s+2(d(h)-s)+d(g)-s\geq d(f)-s+d(g)-s\geq
D(f,g)$.

\noindent{\bf Case II.} Suppose that $\di(f,h)\neq \di(h,g)$.
Without loss of generality we assume that $\di(f,h)> \di(h,g)$. It
follows that $\di(f,g)=\di(h,g)=s$ and that $(f,g)$ and $(h,g)$
are of the same type.

(a) Suppose that $(f,g)$ and $(h,g)$ are of type 1. Then
$a_f(g)=a_h(g)=a_1, a_g(h)=a_g(f) =a_2, y_f(g)=y_h(g) =y_1,
y_g(h)=y_g(f)=y_2$.

(a.1) Suppose that $(f,h)$ is also of type 1. We have
$D(f,h)+D(h,g)=d(f)-a_h(f)+d(h)-a_f(h)+\dist(y_h(f),
y_f(h))+d(h)-a_1+d(g)-a_2+\dist(y_1,y_2)\geq
d(f)-a_1+\dist(y_1,y_2)+d(g)-a_2=D(f,g).$

(a.2) Suppose that $(f,h)$ is of type 2. Then
$D(f,h)+D(h,g)=d(f)-\di(f,h)+d(h)-\di(f,h)+d(h)-a_1+d(g)-a_2+\dist(y_1,y_2)\geq
d(f)-\di(f,h) +\di(f,h)-a_1+d(g)-a_2)+\dist(y_1,y_2)=D(f,g)$.

(b) Suppose that $(f,g)$ and $(h,g)$ are of type 2.

(b.1) Suppose that $(f,h)$ is of type 1. We have
$D(f,h)+D(h,g)=d(f)-a_h(f)+d(h)-a_f(h)+\dist(y_h(f), y_f(h))+
d(h)-s+d(g)-s\geq d(f)-a_h(f)+d(y_h(f),
y_f(h))+a_f(h)-s+d(g)-s\geq d(f)-s+d(g)-s=D(f,g)$.

(b.2) Suppose that $(f,h)$ is of type 2. Then $D(f,h)+D(h,g)=d(f)
- \di(f,h)+d(h)- \di(f,h)+d(h)-s+d(g)-s\geq d(f) -
\di(f,h)+\di(f,h)-s+d(g)-s=D(f,g)$.
\end{proof}

%%%%%%%%%%%%%%%%%%%%%%%%%%%%%%%%%%%%%%%%%%%%%%%%%%%%%%%%%%%%%%%%%%%%%%%%%%%%%%%%%%%%%%%%%%%%%%%%%%%%%%%%%%%%%%%%%

\section{Geometric structure of tree products}
\label{sec3}
%%%%%%%%%%%%%%%%%%%%%%%%%%%%%%%%%%%%%%%%%%%%%%%%%%%%%%%%%%%%%%%%%%%%%%%%%%%%%%%%%%%%%%%%%%%%%%%%%%%%%%%%%%%%%%%%%

The main goal of this section is to prove the following.

\begin{thm}\label{tpmain}
Let $\pp=\{M_i\mid i\in I\}$ be a set of complete geodesic metric
spaces. Then the tree-product $\Pi\pp$ is a complete, homogeneous tree-graded metric space, where every piece is
isometric to one of the $M_\alpha\in\pp$ and every transversal tree $T_x$ is isometric to $\ttt$.
\end{thm}

The proof is divided into a sequence of lemmas.

\begin{lem}\label{complete}
$\Pi\pp $ is complete.
\end{lem}

\begin{proof}
Let $(f_n)$ be a Cauchy sequence in $\Pi\pp$, let $A_n$ be
the domain of the function $f_n$. Let $T_{nm}=\di(f_n,f_m)$ if
$(f_n,f_m)$ is of type 2 and $T_{nm}=a_{f_n}(f_m)$ if $(f_n,f_m)$
is of type 1. Let $T_n=\inf_{m\geq n} T_{nm}$, thus $(T_n)$ is non-decreasing.

The first inequality in Lemma \ref{D}, (1), implies that
$(d(f_n))$ is a Cauchy sequence in $\R$ therefore a convergent
sequence with some limit $d$. If $d=0$ then
$\lm_{n\to\infty}D(f_n, f_{\emptyset})=\lm_{n\to \infty} d(f_n)=
0$ hence $\lm_{n\to\infty}f_n = f_{\emptyset}$. So let us suppose
that $d>0$.

We denote $\varepsilon_n=\sup_{k,m\geq n} D(f_k,f_m)$. Since
$(f_n)$ is a Cauchy sequence $(\varepsilon_n)$ is a decreasing
sequence converging to $0$. By definition $D(f_n,f_m)\geq
d(f_n)-T_{nm}\ge d(f_n)-T_n$. We conclude that $d(f_n)-T_{n}\leq
\varepsilon_n $. It follows that $\lm_{n\to\infty}T_n = d$. It
also follows that the restrictions $g_n=f_n|_{]0,T_n[}$ also form
a Cauchy sequence and $\lm_{n\to\infty}D(g_n,f_n)=\lm_{n\to
\infty}T_n-d(f_n)=0$, so that it suffices to prove that the
sequence $(g_n)$ converges.

By definition we have that for every $m\geq n$ either
$g_n=f_m|_{]0,T_n[}$ or $(g_n,f_m)$ is of type 1 and
$T_n=a_{g_n}(f_m)$. If the latter case ever occurs then there
exists $s_n<T_n$ such that $g_n=\bar{g}_n \sqcup
\mathbf{C}_{(x_n,y_n)}$, where $\bar{g}_n=g_n|_{]0,s_n[}$, and for
every $m>n$ such that $(g_n,f_m)$ is of type 1,
$f_m|_{[0,\sigma_{f_m}(g_n)]}=\bar{g}_n \sqcup
\mathbf{C}_{(x_n,y_m)}$. If the latter case ever occurs then we put
$s_n=T_n$.

The statements in the previous paragraph remain true if we replace
$f_m$ with $g_m$. Consequently the sequence $(s_n)$ is increasing.
On the other hand $s_n\leq T_n$, hence $(s_n)$ is convergent.

If $\lm_{n\to\infty}s_n = d$ then  $D(\bar{g}_n,g_n)=T_n-s_n$
converges to $0$ so we can replace the functions $g_n$ with
$\bar{g}_n$. For $m\geq n$ we have that $\bar{g}_m$ is an
extension of $\bar{g}_n$. We consider $A=\cup_n \left[ A_n\cap
[0,s_n] \right]$ and $g$ defined on $A$ by
$g|_{[0,s_n]}=\bar{g}_n$. The function $g$ is the limit of the
sequence $\bar{g}_n$.

Suppose that $s_n$ converges to $s<d$. Then $T_n-s_n$ converges to
$\delta = d-s>0$. Let $n_0$ be such that for all $n\geq n_0$,
$\varepsilon_n < \frac{\delta}{2}$ and $T_n-s_n >
\frac{\delta}{2}$. In particular for all $k,m\geq n_0$,
$D(g_k,g_m)\leq D(f_k,f_m)< \frac{\delta}{2}$. Let $n\geq n_0$
fixed. For every $m\geq n$ the function $g_m$ extends $g_n$ which
by the previous hypotheses implies $g_m=g_n$ or $(g_m,g_n)$ is of
type 1 which, in turn, implies that $g_m=\bar{g}_n\sqcup
\mathbf{C}_{(x_n,y_m)}$. We then have that for $n\geq n_0$ the
sequence $s_n$ is constant equal to $s$, the sequence $x_n$ is
constant equal to $x$,  $A(f_n)\cap ]0,s[$ is constant equal to
$A$, and that $(y_n)$ is a Cauchy sequence in some $M_\alpha$ or
in $\ttt$. Since $\ttt$ and all $M_\alpha$ are complete, $y_n$
converges to a point $y$. We define $g=\bar{g}_n\sqcup
\mathbf{C}_{(x,y)}$ for every $n\geq n_0$. Since for $n\geq n_0$
we have $D(g_n,g)=d(y_n,y)$ we conclude that $g_n\to g$.
\end{proof}

\begin{lem} \label{homogeneous}
$\Pi\pp$ is homogeneous.
\end{lem}

\proof Let $f$ be an arbitrary function in $\Pi\pp$. We construct
an isometry $\Phi_f$ such that $\Phi_f(f)=f_{\emptyset }$.

We denote $\hat{f}=\overline{f\circ \sigma_{d(f)/2}}$, which is a
function in $\Pi\pp$ (see Remark \ref{backward}).

Let $g$ be an arbitrary function in $\Pi\pp\setminus \{f\}$. If
$(f,g)$ is of type 1 we write
$$\Phi_f(g)=\hat{f}|_{]0,d(f)-a_g(f)[}\sqcup
\mathbf{C}_{(y_g(f),y_f(g))}\sqcup
 g|^\iota_{]a_f(g),d(g)[}.$$

If $(f,g)$ is of type 2 we write $$\Phi_f(g)=
\hat{f}|_{]0,d(f)-\di(f,g)[}\sqcup g|^\iota_{]\di(f,g),d(g)[}.$$

It is easy to check that $\Phi_f(g)$ satisfies all the conditions
$\Pi_1-\Pi_5$, so $\Phi_f(g)\in\Pi\pp$.

The proof that $\Phi_f$ preserves the metric $D$ follows the same
path as the proof of the triangle inequality for $D$ (one has to consider cases as in the proof of Proposition \ref{trin}, and the treatment of each case is similar; the only significant difference is that one has to use Properties $\Pi_4$ and $\Pi_5$ here), so we leave
it as an exercise.

To conclude the proof, we show that $\Phi_f$ is surjective. Take
$g\in \Pi\pp$. First suppose that $(g, \hat{f})$ is of type 1.
Then let $$h=f|_{]0,a_g(\hat{f})[}\sqcup \mathbf{C}_{(y_g(\hat
f),y_{\hat f}(g))} \sqcup g|^\iota_{]d(g)-a_{\hat f}(g),d(g)[}.$$
It is easy to see that $h\in\Pi\pp$. Notice that $(f,h)$ is of
type 1, $\di(f,h)=d(f)-a_g(\hat f)$, $a_h(f)=d(f)-\di(\hat{f},g)$,
$a_f(h)=d(f)-a_g(\hat f)+\dist(y_g(\hat f),y_{\hat f}(g))$,
$y_h(f)=x_g(\hat f)=x_{\hat f}(g)$, $y_f(h)=y_{\hat f}(g)$,
$d(h)=a_f(h)+d(g)-a_g(\hat f)$. So

$$\begin{array}{ll}\Phi_f(h) & =\hat{f}|_{]0,d(f)-a_h(f)[}\sqcup
\mathbf{C}_{(y_h(f),y_f(h))}\sqcup
 h|^\iota_{]a_f(h),d(h)[}\\& =
 \hat{f}|_{]0,\di(\hat f,g)[}\sqcup \mathbf{C}_{(x_{\hat f}(g),y_{\hat f}(g))}
 \sqcup g|^\iota_{]d(g)-a_{\hat f}(g),d(g)[} = g.
\end{array}
$$

Now suppose that $(g,\hat{f})$ is of type 2. Let
$$h=f|_{]0,d(f)-\di(\hat f,g)[}\sqcup g|^\iota_{]\di(\hat f,
g),d(g)[}.$$

Then again it is easy to check that $h\in\Pi\pp$, and $(f,h)$ is
of type 2, $\di(f,h)=d(f)-\di(\hat f, g)$,
$d(h)=d(g)+d(f)-2\di(\hat f,g)$. So

$$\begin{array}{ll}\Phi_f(h) &=
\hat{f}|_{]0,d(f)-\di(f,h)[}\sqcup h|^\iota_{]\di(f,h),d(h)[}\\
&=g|_{]0,\di(\hat f,g)[}\sqcup g|^\iota_{]\di(\hat f,g),
d(g)[}=g.\end{array}$$
\endproof

Now we are going to describe all geodesics joining two points in
$\Pi\pp$. By homogeneity (Lemma \ref{homogeneous}) it suffices to
consider only geodesics connecting $f\in \Pi\pp$ with
$f_{\emptyset}$.

\begin{rem}\label{iv} The isometry $\Phi_f$ and the construction of the function
$h=\Phi_f\iv(g)$ will be used below. In particular, suppose that
for some $\alpha\in I$, $x_\alpha\in M_\alpha$, there exists $y_0
\in M_\alpha \setminus \{ x_\alpha \}$ and $f_0\in \Pi\pp$ such
that $f=f_0\sqcup \mathbf{C}_{(y_0,x_\alpha )}$. Then for every
$y\ne y_0\in M_\alpha$ the function $\Phi_f\iv(\mathbf{C}_{y_0,y})$ is
$f_0\sqcup \mathbf{C}_{(y_0,y)}$ and is $f_0$ for $y=y_0$.

If the $y_0$ does not exist, then
$\Phi_f\iv(\mathbf{C}_{(y_0,y)})$ is $f\sqcup
\mathbf{C}_{(x_\alpha ,y )}$ for every $y\neq x_\alpha\in
M_\alpha$ and is $f$ for $y=x_\alpha$.
\end{rem}

\begin{defn}[geodesics connecting points in $\Pi\pp$ with $f_\emptyset$]\label{geod} For every $f\in\Pi\pp$ consider the set of functions
${\mathfrak g}:[0,d(f)]\to \Pi\pp$ defined as follows. Let
${\mathfrak g}(0)=f_{\emptyset}$. If $t\in ]0,d(f)]\setminus A_f$
we put ${\mathfrak g}(t)=f|_{]0,t[}$. Let $]a,b[$ be a maximal
open interval contained in $A_f$. Then
$f|^\iota_{]a,b[}=\mathbf{C}_{(x,y)}$ for some $(x,y)\in \Omega$.
We have $x,y\in M$, where $M=M_\alpha$. Let $\mathfrak h$ be a
geodesic joining $x$ and $y$ in $M$. We write ${\mathfrak
g}(t)=f|_{]0,a[}\sqcup \mathbf{C}_{(x,{\mathfrak h}(t-a))}$ for
every $t\in ]a,b[$.
\end{defn}

\begin{lem} \label{geodesic} Each such function
$\mathfrak g$ is a geodesic in $\Pi\pp$ connecting $f_\emptyset$
and $f$.
\end{lem}

\proof Let $t,s\in [0,d(f)]$. We need to show that
\begin{equation}\label{geo}
D({\mathfrak g}(t),{\mathfrak g}(s))=|t-s|.
\end{equation}

If neither $t$ nor $s$ is in $A_f$ then $D({\mathfrak
g}(t),{\mathfrak g}(s))=D(f|_{]0,t[},f|_{]0,s[})=|t-s|$ by the
definition of $D$.

Suppose that $t\in A_f$ and $s \not\in A_f$. Let $]a,b[\subset
A_f$ be the maximal open sub-interval of $A_f$ containing $t$ and
let $f|^\iota_{]a,b[}=\mathbf{C}_{(x,y)}$. We have ${\mathfrak
g}(t)=f|_{]0,a[}\sqcup \mathbf{C}_{(x,{\mathfrak h}(t-a))}$, where
$\mathfrak h$ is a geodesic joining $x$ and $y$. If $s<t$ then the
function ${\mathfrak g}(t)$ is an extension of the function
${\mathfrak g}(s)$ and the equality (\ref{geo}) is satisfied.

Suppose that $t<s$. Then $({\mathfrak g}(t),{\mathfrak g}(s))$ is
of type 1 and $D({\mathfrak g}(t),{\mathfrak
g}(s))=\dist({\mathfrak h}(t-a),y)+s-b=b-t+s-b=s-t$.

Finally suppose that $t,s\in A_f$. If $t$ and $s$ are both
contained in $]a,b[\subset A_f$ maximal open interval then the
equality (\ref{geo}) is obvious. If not then there exists a number
$r \not\in A_f$ between $t$ and $s$. The previous argument applied
to $]t,r[$ and to $]r ,s[$ implies $D({\mathfrak g}(t),{\mathfrak
g}(s)) \leq |t-s|$. The converse inequality follows from Lemma
\ref{D}(1).
\endproof

\begin{lem} \label{every} Every geodesic connecting $f_\emptyset$ to
$f$ is equal to a geodesic defined as in Definition \ref{geod}.
\end{lem}

\proof Let ${\mathfrak g}\colon\, ]0,d(f)[\to \Pi\pp$ be any
geodesic connecting $f_\emptyset$ and $f$. For every $t\in ]0,d(f)[$
let $f_t={\mathfrak g}(t)$. Since $D(f_{\emptyset },f_t)=t$ it
follows that the domain of definition $A_t$ of $f_t$ is an open
dense set in $]0,t[$.

Suppose that $(f,f_t)$ is of type 1. Then $D(f,f_t)=
d(f)-t+2(t-a_f{f_t})+\dist(x,y_f(f_t)+\dist(y_f(f_t),y_{t})-\dist(x,y_f)$.
Since $D(f,f_t)= d(f)-t$ it follows that $t=a_f(f_t)$ and that
$\dist(x,y_{t})+\dist(y_f,y_{t})=\dist(x,y_f)$. In particular it
follows that $$\di(f,f_t)<t=\di(f,f_t)+\dist(x,y_{t})<a_{f_t}(f)$$
and $]\di(f,f_t),a_{f_t}(f)[$ is an interval on which $f$ is
constant. Hence $t\in A_f$.

Suppose that $(f,f_t)$ is of type 2. Then $D(f,f_t)=
d(f)-t+2(t-\di(f_t,f))$ and as it is equal to $d(f)-t$ it follows
$t=\di(f_t,f)\not\in A_f$ and $f_t=f|_{]0,t[}$.

We conclude that $(f,f_t)$ is of type 1 if and only if $t\in A_f$.

Let $]a,b[$ be a maximal open interval contained in $A_f$. We have
$f|_{]0,b[}=f|_{]0,a[}\sqcup \mathbf{C}_{(x,y)}, x,y\in M$ and by
the previous argument for every $t\in ]a,b[$,
$f_t=f|_{]0,a[}\sqcup \mathbf{C}_{(x,y_t)}$. Since for every
$t,s\in ]a,b[$ we have $D(f_s,f_t)=\dist(y_s,y_{t})=|t-s|$ it
follows that $s\to y_s$ is an isometric embedding of $]a,b[$ in
$M$. As $D(f_a,f_t)=d(x,y_t)=t-a$ and $D(f_t,f_b)=d(y_t,y)=b-t$ it
follows that $\lm_{t\to a}y_t=x$ and that $\lm_{t\to b}y_t=y$. We
conclude that $s\to y_s$ is a geodesic from $x$ to $y$.
\endproof

In particular, Lemma \ref{every} implies that is every element of $\pp $ is uniquely geodesic, then so is $\Pi\pp$.

\begin{lem}\label{tgstructure}
$\Pi\pp$ is a tree-graded space where every piece is
isometric to one of the $M_\alpha\in\pp$ and every transversal tree $T_x$ is
isometric to $\ttt$.
\end{lem}

\begin{proof} We start by defining the pieces containing the point
$f_\emptyset$. Other pieces will be defined using the homogeneity.

For every $\alpha \in I$ (note: $\alpha\ne 0$) and $x_\alpha \in
M_\alpha $, the piece
$$
M_\alpha(f_{\emptyset },x_\alpha )=\left\{ \mathbf{C}_{(x_\alpha
,y )}\mid y\in M_\alpha \setminus \{ x_\alpha \}\right\}\cup \{
f_{\emptyset} \}.
$$
Clearly the map $\Psi_\alpha(x_\alpha)\colon M_\alpha \to
M_\alpha(f_{\emptyset },x_\alpha )$ that takes $x_\alpha
\rightarrow f_{\emptyset }$ and $y\rightarrow
\mathbf{C}_{(x_\alpha ,y )}$ for every $y\in M_\alpha \setminus \{
x_\alpha \}$ is an isometry.

For every other point $f\in \Pi\pp$ the piece
$M_\alpha(f,x_\alpha)$ containing $f$ is defined as
$\Phi_f\iv(M_\alpha(f_\emptyset,x_\alpha))$ where $\Phi_f$ is the
isometry defined in the proof of Lemma \ref{homogeneous}.

Notice that for every $\alpha\in I$, every $x_\alpha\in M_\alpha$
and every $f\in \Pi\pp$ we associate a unique isometric copy
$M_\alpha(f,x_\alpha)$ of $M_\alpha$ in $\Pi\pp$ and an isometry
$\Phi_f\iv\Psi_\alpha(x_\alpha)$ from $M_\alpha$ to
$M_\alpha(f,x_\alpha)$ that takes $x_\alpha$ to $f$. In other
words, for every $f\in \Pi\pp$ and every $x_\alpha\in M_\alpha$
there exists a copy of $M_\alpha$ in $\Pi\pp$ attached to $f$ by
$x_\alpha$.

Now let us show that different pieces have at most one point in
common (property ($T_1$) of tree-graded spaces). Suppose
$M_\alpha(f_{\emptyset },x_\alpha )\cap M_\beta(f,x_\beta )$
contains two distinct functions $g,h$. By Remark \ref{iv} we may
suppose, up to changing $f$ and $x_\beta$, that $M_\beta(f,x_\beta
)$ consists of the functions $f\sqcup \mathbf{C}_{(x_\beta ,y )}$
for every $y\neq x_\beta $ and of the function $f$. Since $g,h\in
M_\alpha(f_{\emptyset },x_\alpha )$ at least one of the two is of
the form $\mathbf{C}_{(x_\alpha ,x)}, x\neq x_\alpha$. Suppose it
is $g$.

If $h=f_{\emptyset}$ then $h\in M_\beta(f,x_\beta )$ implies
$h=f$. On the other hand $g=\mathbf{C}_{(x_\alpha ,x)}\in
M_\beta(f,x_\beta )$ implies $\alpha=\beta$ and $x_\alpha=x_\beta$
and hence $M_\alpha(f_{\emptyset },x_\alpha )= M_\beta(f,x_\beta
)$.

Suppose $h=\mathbf{C}_{(x_\alpha ,y)}\, ,\, y\neq x_\alpha$. Since
$h\in M_\beta(f,x_\beta )$ either $h=f$ or $h=f\sqcup
\mathbf{C}_{(x_\beta ,z)}$. In the latter case it follows that
$f=f_{\emptyset}$, $\alpha=\beta$ and $x_\alpha=x_\beta$ and we
can conclude. Suppose that $h=f$. Then $g=\mathbf{C}_{(x_\alpha
,y)}\sqcup \mathbf{C}_{(x_\beta ,u)}$, contradiction.

Now let us show that every simple geodesic triangle in $\Pi\pp$ is
contained in one piece (Property ($T_2$)). First we prove that for
geodesic simple bigons. It suffices to consider a bigon with
vertices $f_{\emptyset}$ and $f\in \Pi\pp\setminus \{ f_{\emptyset}
\}$. Let ${\mathfrak g}_1$ and ${\mathfrak g}_2$ be the two
geodesics forming this bigon. By Lemma \ref{every} $A_f=\, ]0,d(f)[$
\,: otherwise the bigon would not be simple. Hence
$f=\mathbf{C}_{(x,y)}$ for some $(x,y)\in M\times M$.

Again Lemma \ref{every} implies that
$\g_1(t)=\mathbf{C}_{(x,\pgot(t))}$ and
$\g_2(t)=\mathbf{C}_{(x,\rrr(t))}$ such that $\pgot$ and $\rrr$
are geodesics joining $x$ and $y$ in $M$. If $M=\ttt$ then $\g_1 =
\g_2$, a contradiction. Therefore $M=M_\alpha$, $\alpha\ne 0$, and
$\g_1 \cup \g_2 \subset M_\alpha(f_{\emptyset },x )$ as required.

Now consider a geodesic simple triangle $\Delta = \g_1\cup
\g_2\cup \g_3$. By homogeneity we may suppose that the common
endpoint of the geodesics $\g_1$ and $\g_2$ is $f_{\emptyset}$.
Let $f_1$ and $f_2$ be their other respective endpoints.

If $s=\di(f_1,f_2)>0$ then $s\not \in A_{f_1}\cup A_{f_2}$ and
${\mathfrak g}_1(s)=f_1|_{]0,s[}=f_2|_{]0,s[}={\mathfrak g}_2(s)$,
contradiction. Therefore $s=0$.

Suppose that $(f_1,f_2)$ is of type 2. Then
$D(f_1,f_2)=d(f_1)+d(f_2)$ and $\g_1\cup \g_2$ is a geodesic
between $f_1$ and $f_2$. It remains to apply the result about
bigons.

Suppose that $(f_1,f_2)$ is of type 1. Then, as $s=0$, there exist
$a_i>0 $ and $(x,y_i)\in M \times M$ such that
$f_i|_{]0,\sigma_i[}=\mathbf{C}_{(x,y_i)}$ for $i=1,2$. We have
\begin{equation}\label{fd}
D(f_1,f_2)=d(f_1)-a_1+d({f_2})-a_2+d(y_1,y_2).
\end{equation}

According to Lemma \ref{every}, $a_i\not\in A_{f_i}$ implies
$\g_i(a_i)=f_i|_{]0,a_i[}$ for $i=1,2$. We denote by $\g_i'$ the
sub-arc of $\g_i$ with endpoints $\g_i(a_i)$ and $f_i$.

Consider a geodesic $\pgot$ between $y_1$ and $y_2$ in $M$ and let
$\g(t)=\mathbf{C}_{(x,\pgot(t))}$ be the corresponding geodesic
between $f_1|_{]0,\sigma_1[}$ and $f_2|_{]0,\sigma_2[}$. Formula
(\ref{fd}) implies that $\g_3'=\g_1'\cup \g \cup \g_2'$ is a
geodesic joining $f_1$ and $f_2$.

If $\g_3'\cup \g_3$ is a simple bigon then it is contained in a
piece, $M_\alpha (f,x_\alpha )$. In particular
$\mathbf{C}_{(x,y_i)}, i=1,2$, are contained in $M_\alpha
(f,x_\alpha)$.

By Remark \ref{iv} we may suppose that
$\Phi_f\iv(\mathbf{C}_{(y_0,y)})$ is $f\sqcup
\mathbf{C}_{(x_\alpha ,y )}$ for every $y\neq x_\alpha\in
M_\alpha$ and is $f$ for $y=x_\alpha$.

Since $y_1\neq y_2$ it follows that at least one of the two
functions $\mathbf{C}_{(x,y_i)}, i=1,2$, is of the form $f\sqcup
\mathbf{C}_{(x_\alpha ,z)}$. Consequently, $f=f_{\emptyset }$ and
$(x,y_i)=(x_\alpha ,z)\in \Omega_{\alpha }$. In particular
$M=M_\alpha $, so ${\mathfrak g} \subset
M_\alpha(f_{\emptyset},x)$. It follows that the whole bigon is
contained in this piece. We conclude that $\Delta$ is also
contained in this piece.

Suppose ${\mathfrak g}_3'\cap {\mathfrak g}_3 \neq \{ f_1,f_2 \}$.
Since $\Delta$ is simple this implies ${\mathfrak g}\cap
{\mathfrak g}_3 \neq \emptyset$. Let $g$ be the nearest point of
$f_1$ which ${\mathfrak g}_3 $ has in common with ${\mathfrak g}$.
If $g= \mathbf{C}_{(x,y_1)}$ the simplicity of $\Delta$ implies
$g=f_1$. If $g\neq \mathbf{C}_{(x,y_1)}$ then the sub-arcs of
${\mathfrak g}_3$ and ${\mathfrak g}_3'$ with endpoints $f_1$ and
$g$ form a simple bigon intersecting ${\mathfrak g}$ in a
non-trivial sub-arc. By the previous argument it follows that the
whole bigon is in the piece $M_\alpha (f_{\emptyset},x)$,
therefore $f_1=\mathbf{C}_{(x,y_1)}$.

A similar argument implies that $f_2=\mathbf{C}_{(x,y_2)}$. A
decomposition of ${\mathfrak g}_3'\cup {\mathfrak g}_3$ into
simple bigons and geodesics joining their endpoints allows us to
conclude that ${\mathfrak g}_3'\cup {\mathfrak g}_3\subset
M_\alpha (f_{\emptyset},x)$, therefore $\Delta \subset M_\alpha
(f_{\emptyset},x)$.

It remains to prove that for every $f\in \Pi\pp$ the tree $T_f$ is
isometric to $\ttt$. It suffices to consider $f=f_{\emptyset }$.
The set $T_{\emptyset}$ of topological arcs with one endpoint
$f_{\emptyset }$ intersecting every piece in at most one point is
the same as the set of points in $\Pi\pp$ which can be joined to
$f_{\emptyset}$ by such arcs. By Lemma \ref{every} this set is an
$\R$-tree which is a union of trees isometric to $\ttt$. By the
uniqueness result from \cite{MNO,EP}, we conclude that $T_\emptyset$
is isometric to $\ttt$.
\end{proof}

%%%%%%%%%%%%%%%%%%%%%%%%%%%%%%%%%%%%%%%%%%%%%%%%%%%%%%%%%%%%%%%%%%%

\section{Universal $\qq$-trees}
\label{5}
%%%%%%%%%%%%%%%%%%%%%%%%%%%%%%%%%%%%%%%%%%%%%%%%%%%%%%%%%%%%%%%%%%%

Let $\qq $ be a collection of geodesic metric spaces. Recall that
all metric spaces in this paper are supposed to have cardinality at
most continuum.

%\begin{defn}[{\bf $\mathcal Q$-trees}] A geodesic metric space $\free$ is called a {\it $\mathcal Q$-tree} if $\free$ %is tree-graded with respect to some set of pieces $\pp$ and every piece $P\in \pp$ is isometric to some $Q\in \mathcal %Q$ ($P$ is called a $Q$-piece in this case).
%In what follows we also assume that elements of $\mathcal Q$ are not single points and are pairwise non-isometric.
%\end{defn}

\begin{defn}[{\bf Canonical isometries}]
Let $\free$ be a $\mathcal Q$-tree (see Definition \ref{def1}). For every $Q\in \qq$, we fix a set $D(Q)$ of orbit representatives of the diagonal action of ${\rm Isom } (Q)$ on $Q\times Q$. We also fix a set of orbit representatives $E(Q)$ of the action of ${\rm Isom } (Q)$ on $Q$. For every piece $P$ of $\free$ and every pair of points $(x,y)\in P\times P$ (resp. a point $x\in P$), we choose $Q\in \qq$ and an isometry $\iota \colon P\to Q$ such that $(\iota (x), \iota (y))\in D(Q)$ (resp. $\iota(x)\in E(Q)$). In what follows this isometry is called the {\it canonical} isometry associated to the triple $(P,x,y)$ (pair $(P,x)$). If a pair $(P,x)$ is canonically isometric to $(Q,y)$, then we say that $P$ is a {$(Q,y)$-piece at $x$}.
\end{defn}

\begin{defn}[{\bf Types over $\mathcal Q$}]
Let $U\subset ]0, \infty[$ be a (possibly empty) disjoint union of
bounded intervals
\begin{equation}\label{u}
U=\bigsqcup\limits_{\alpha \in A} ]a_\alpha ,b_\alpha[,
\end{equation}
where $a_\alpha , b_\alpha > 0$. Let also $f$ be a map that assigns
to each interval $]a_\alpha , b_\alpha[$ a  pair $(p,q)\in D(Q)$ for
some $Q\in \mathcal Q$ with $\d (p,q)=b_\alpha -a_\alpha$. We call
such pairs $(U,f)$ {\it $\mathcal Q$-types}. $(U,f)$ is called {\it
trivial} if $U=\emptyset$.
\end{defn}

\begin{defn}[{\bf Equivalent types}]
We say that $(U_0,f_0)$ is an {\it initial subtype} of a $\mathcal
Q$-type $(U,f)$ if $U_0=U\cap ]0,r[$ for some $r\notin U$, and
$f_0\equiv f\vert _{U_0}$. Two $\mathcal Q$-types are {\it
equivalent} if they have equal nontrivial initial subtypes.
\end{defn}

\begin{defn}[{\bf Limit connected components and their types}]
Let $s$ be a point in a \lqt $\free$. Let $C$ be a connected
component of $\free\setminus \{s\} $, $\gamma \colon [0, r]\to
C\cup\{ s\}$ be any geodesic parameterized by length such that
$\gamma (0)=s$. If there exists a piece $P$ of $\free$ such that
$P\cap \gamma$ is a nontrivial initial subsegment of $\gamma$, we
say that $C$ is a {\it non-limit} component. Otherwise we say that
$C$ is a {\it limit component} and define its {\it type} as follows.
Let $U$ be as in (\ref{u}), where for every $\alpha \in A$, $]\gamma
(a_\alpha), \gamma (b_\alpha )[=P_\alpha \cap \gamma $ for some
piece $P_\alpha$ of $\free$ and any piece $P\notin \{ P_\alpha \mid
\alpha \in A\} $ intersects $\gamma $ in at most $1$ point. Note
that $U\ne ]0,r[$ since in (\ref{u}), $a_\alpha>0$ by our
assumption. The map $f$ is defined by the rule $f]a_\alpha ,
b_\alpha[= (\iota (\gamma (a_\alpha)), \iota (\gamma (b_\alpha)))$,
where $\iota $ is the canonical isometry associated to the triple
$(P_\alpha, \gamma (a_\alpha), \gamma (b_\alpha ))$. We say that
$(U,f)$ is the {\it type of $\gamma $} and  define the {\it type of
$C$}, denoted $\tau (C)$, to be the equivalence class of $(U,f)$.
\end{defn}

It turns out that the notions of a limit and non-limit component and types of limit ones are well-defined, i.e., they are independent of the choice of a particular geodesic $\gamma $. The following well-known fact will be used several times: any continuous image of a closed interval in a Hausdorff space is arc-connected (i.e., connected by injective paths).

\begin{lem}\label{wd}
Let $s$ be a point in a \lqt $\free$ and let $C$ be a connected component of $\free\setminus \{s\} $. Then at least one of the following holds.
\begin{enumerate}
\item[(a)] There exists a unique piece $P$ of $\free$ such that for every $x\in C$, and every geodesic $[s,x]$ in $C\cup\{ s\}$, the intersection $P\cap [s,x]$ is a nontrivial initial subsegment of $[s,x]$.
\item[(b)] For any $x,y\in C$, any two geodesics $[s,x], [s,y]$ in $C\cup\{ s\}$ have at least one common point $u\ne s$. In particular, $[s,x]$ and $[s,y]$ have equivalent types.
\end{enumerate}
\end{lem}

\begin{proof}
Let us take any $x,y\in C$ and consider any geodesics $[s,x], [s,y]$ in $C\cup\{ s\}$. Let us also fix any geodesics $[x,y]\subseteq C$. Since both $[s,x]$ and $[x,y]$ are compact, there exists $u\in [s,x]$, $u\ne s$, such that the subsegment $[s,u]$ of $[s,x]$ intersects $[x,y]$ trivially. Since the concatenation $[s,x][x,y]$ is arc-connected, there exists an arc $[u,y]\subseteq [s,x][x,y]$. Clearly the concatenation $c=[s,u][u,y]$ is also an arc. Now there are two cases to consider.

First assume that there is a piece $P$ which contains a nontrivial initial subsegment of $[z,y]$. By Lemma \ref{DS}, $P$ also contains a nontrivial initial subsegment of $c$.  Since $c$ contains a nontrivial initial subsegment of $[s, x]$, $P\cap [s,x]$ is a nontrivial initial subsegment of $[s,x]$. Since $[s,x]$ was arbitrary, we obtain (a).

Now suppose that such a piece $P$ does not exist. Then again by Lemma \ref{DS} $]s,x]$ and $c$ have common points arbitrary close to $s$. Then $]s,x]$ has a common point with $]s,y]$ and we get (b). The statement ``in particular" follows immediately from Lemma \ref{DS}.
\end{proof}

\begin{defn}[{\bf Universal $\mathcal Q$-trees}]\label{defu}
Let $\T $ denote the set of all possible equivalence classes of $\mathcal Q$-types.
A  {\it $\mathcal Q$-tree} $\free$ is called {\it universal} if the following conditions hold:
\begin{enumerate}
\item[(a)] $\free$ is complete;

\item[(b)] for any $s\in \free$, any $\theta \in \Theta$ there are exactly continuum limit connected components of $\free\setminus \{ s\} $ of type $\theta $;

\item[(c)] for any $Q\in \mathcal Q$, and any $x\in E(Q)$  and exactly continuum $(Q,x)$-pieces of $\free$ containing $s$.
\end{enumerate}
\end{defn}

{\it Remark. } The above definition can be generalized by replacing
continuum with any other cardinal number (or even two different
cardinals in (b) and (c)) as long as these cardinals are larger
than the cardinality of the set $\mathcal Q$. We do not develop the
general theory since the particular case is sufficient for our
applications.

\begin{defn}[{\bf Good $\mathcal Q$-subtrees}]
Let $\free$ be a $\mathcal Q$-tree, $\free^\prime$ a subset of $\free$. We say that $\free^\prime$ is a {\it $\mathcal Q$-subtree} of $\free$ if $\free^\prime$ is path-connected and every piece $P$ of $\free$ that intersects $\free^\prime$ in at least $2$ points, is contained in $\free'$. Further we say that $\free^\prime$ is {\it good} if for every point $t\in \free^\prime $ one of the following conditions holds:
\begin{enumerate}
\item[(a)] $\free^\prime $ intersects nontrivially every connected component of $\free\setminus \{ t\}$. In this case $t$ is called a  {\it point of type I}.
\item[(b)] $\free^\prime $ intersects nontrivially at most $2$ limit connected components of $\free\setminus \{ t\}$ and at most $2$ pieces of $\free$ containing $t$. In this case $t$ is called a  {\it point of type II}.
\end{enumerate}
\end{defn}

\begin{lem}\label{int}
Let $\free^\prime$ be a \lqst of a \lqt $\free$. Then the following hold.
\begin{enumerate}
\item[(a)] $\free^\prime $ is totally geodesic in $\free$, i.e., for any two points $x,y\in \free^\prime $,  $\free^\prime$ contains all geodesics connecting $x$ to $y$ in $\free$. In particular, $\free^\prime $ is geodesic.
\item[(b)] For any $x\in \free$ and any $y\in \free^\prime$, the intersection of any geodesic $[y,x]$ and $\free^\prime$ is either a subsegment $[y,z]$ or a half-open subinterval $[y,z)$ of $[y,x]$.
\end{enumerate}
\end{lem}

\begin{proof}
(a) Since $\free^\prime $ is path-connected, it contains an arc $\gamma ^\prime $ connecting $x$ to $y$. By Lemma \ref{DS}, every geodesic $\gamma $ connecting $x$ and $y$ in $\free$ may differ from $\gamma^\prime $ only inside pieces, which $\gamma ^\prime$ intersects in more than one point. Since all such pieces are contained in $\free^\prime $ by the definition of a \lqst \ $\gamma $ is also contained in $\free^\prime $.

(b) This immediately follows from (a).
\end{proof}

\begin{lem}\label{comp}
Let $\free^\prime $ be a good \lqst of a \lqt $\free$ and let $t$ belongs to $\overline{\free^\prime} $, the closure of $\free^\prime $ in $\free$. Then $\free^\prime \cup\{ t\}$ is also a good \lqst of $\free$.
\end{lem}

\begin{proof}
Let $\freg=\free'\cup\{ t\}$. Recall that $\free$ is geodesic by Lemma \ref{int}. It is straightforward to verify that for any geodesic subspace $M$ of a metric space $N$ and any point $m\in \overline{M}$, the space $M\cup \{m\} $ is path-connected. Hence $\freg$ is path-connected. Further suppose that there is a piece $P$ of $\free$ such that $P\cap \freg$ contains at least two points, say $x$ and $y$. Let $c$ be an arc connecting $x$ to $y$ in $\freg$. By Lemma \ref{DS} $c$ is contained in $P$, thus $P\cap \freg$ is infinite. Thus $P\cap \free^\prime $ contains at least 2 points and hence $P\subseteq \free^\prime \subset \freg$ by the definition of a $\mathcal Q$-subtree. Thus $\freg$ is a \lqst of $\free$.

It remains to show that $\freg$ is good. Let $s\in \freg$. If $s\ne t$ and $s$ was of type I in $\free^\prime$, it remains of type I in $\freg$. Further suppose that $s\ne t$ was of type II in $\free^\prime $. Let $C$ be a limit connected component of $\free\setminus \{ s\} $, or a piece of $\free$ containing $s$. Suppose that $\freg$ intersects $C$ nontrivially. Let $r\in C\cap \freg$ and let $[s,r]$ be a geodesic in $\free$. Then $]s,r]\in C\cap \freg$ and hence $\free^\prime $ intersects $C$ nontrivially. Thus $s$ is of type II in $\freg$. Finally we note that $t$ is clearly of type II in $\freg$ since $\freg\setminus \{ t\} =\free^{\prime }$ is connected.
\end{proof}

The main result of this section is the following. We say that a {\it branching degree} of a $\qq $-tree $\free $ is at most continuum if for any $s\in \free$, any $\theta \in \Theta$, any $Q\in \mathcal Q$, and any $x\in E(Q)$, there are at most continuum limit connected components of $\free\setminus \{ s\} $ of type $\theta $ and at most continuum $(Q,x)$-pieces of $\free$ containing $s$.

\begin{thm}\label{univ} Let $\mathcal Q$ be a collection of homogeneous complete geodesic metric spaces, each of cardinality at most continuum. Then the following hold.
\begin{enumerate}
\item There exists a universal $\qq $-tree, namely the tree product of the collection consisting of continuum isometric copies  of every $Q\in \qq$.

\item Every $\qq$-tree of branching degree at most continuum isometrically embeds into a universal $\qq$-tree.

\item Every two universal $\mathcal Q$-trees are isometric and the isometry preserves pieces.
\end{enumerate}
\end{thm}

 \begin{proof}
(1) For every space $Q\in \qq$ consider continuum isometric copies
of $Q$ and let $\qq'$ be the collection of all these copies. Then
the tree product $\Pi\qq'$ is a universal $\qq$-tree. Indeed, let
$s$ be a point in $\Pi\qq'$. Since $\Pi\qq'$ is homogeneous, we can
assume that $s=f_\emptyset$. Then the fact that for every $Q\in
\qq$, $x\in E(Q)$ there exists continuum $(Q,x)$-pieces containing
$s$ follows immediately from the proof of Lemma \ref{tgstructure}
and the fact that for every limit type $\tau$, there exist continuum
connected components of $\Pi\qq'\setminus \{s\}$ of type $\tau$
follows from Lemma \ref{geodesic} and the definition of the tree
product.

(2) Let $\freg$ be a $\qq$-tree of branching degree at most
continuum, $\free$ a universal $\mathcal Q$-tree. Consider the set
$\Delta $ of all triples $(\freg^\prime, \free^\prime, f)$, where
$\freg^\prime $ and $\free^\prime $ are good $\mathcal Q$-subtrees
of $\freg$ and $\free$, respectively, and $f\colon \freg^\prime \to
\free^\prime $ is an isometric embedding, which preserves the types
of points. $\Delta $ is non-empty since the triple $(\freg^\prime,
\free^\prime, f)$, where $\freg^\prime $ and $\free^\prime $ are
points and $f$ is the obvious map, belongs to $\Delta $.

We equip $\Delta $ with the standard ordering:
$$
(\freg^\prime, \free^\prime, f)\preceq (\freg^{\prime \prime}, \free^{\prime\prime}, g)\;\; {\rm iff} \;\; \freg^\prime\subseteq \freg^{\prime\prime},\; \free^\prime\subseteq \free^{\prime\prime},\; f\equiv g\vert_{\freg^\prime}.
$$
It is easy to show that the union of every chain of elements of $\Delta $ (defined in the standard way) is again an element of $\Delta $. Hence by the Zorn Lemma there exists a maximal triple $(\freg_{\max}, \free_{\max}, f_{\max})$. We are going to show that $\freg_{\max}=\freg$.

First assume that $\freg_{\max} \ne \freg$, i.e., there is $x\in \freg\setminus \freg_{\max} $. Let us fix any $y\in \freg_{\max}$. By Lemma \ref{int}, the intersection of a geodesic $[y,x]$ and $\freg_{\max}$ is either $[y,z]$ or $[y,z)$ for some $z\in [y,x]$. We deal with these cases separately and will arrive at a contradiction in both of them.
\smallskip

{\it Case 1.} $[y,x]\cap \freg_{\max}=[y,z]$. Let us show that $z$ is a vertex of type II in this case. Since $\freg_{\max}$ is good, it suffices to show that $\freg_{\max}\cap C= \emptyset $, where $C$ is the connected component of $\freg_{\max}\setminus \{ z\}$ that contains $x$.

We argue by contradiction. Suppose that there exists $t\in \freg_{\max}\cap C$. Let us fix any geodesic $[z,t]\subseteq C\cup \{ z\}$. If $C$ has a non-limit type, then by Lemma \ref{wd} there is a piece $P$ such that $P\cap [z,t]$ is a nontrivial initial subsegment of $[z,t]$. Since $[z,t]\in \freg_{\max}$, $P$ is contained in $\freg_{\max}$. As $P$ contains a nontrivial initial subsegment of $[z,x]$, this contradicts the choice of $z$. If $C$ has a limit type, then $]z,x]$ has a nontrivial intersection with $]z,t]$ by Lemma \ref{wd} and we get the same contradiction again.

Thus $z$ is a point of type II and hence so is $f_{\max}(z)$. Since $\freg$ and $\free$ are universal and $\freg_{\max}$, $\free_{\max}$ are good, there exists a 1-to-1 correspondence $g$  between connected components of $\freg\setminus (\{ z\}\cup \freg_{\max})$ and $\free\setminus (\{ f_{\max} (z)\}\cup \free_{\max})$ which preserves types. For every limit connected component $E$ of $\freg\setminus (\{ z\}\cup \freg_{\max})$  there exist geodesics $\gamma \subset E\cup\{ z\} $ and $\delta \subset g(E)\cup\{ f_{\max}(z)\} $ of equivalent types starting from $z$ and $f_{\max}(z)$, respectively. Passing to their initial subsegments, we can assume that $\gamma $ and $\delta $ have the same type. We join $\gamma $ and all pieces which intersect $\gamma $ in more than one point to $\freg_{\max}$ and do the same with $\delta $ and $\free_{\max}$. Further we add to $\freg_{\max}$ (respectively, $\free_{\max}$) all pieces of $\freg$ (respectively, $\free$) containing $s$ (respectively, $t$).

Let $\freg_1, \free_1$ be the subspaces of $\freg$ and $\free$ obtained from $\freg_{\max}$ and $\free_{\max}$ in this way. It is straightforward to check that $f_{\max}$ extends to an isometry between $\freg_1$ and $\free_1$. Let us first verify that $\freg_1$ and $\free_1$ are $\mathcal Q$-subtrees of $\freg$ and $\free$, respectively. Clearly they are path-connected. Further suppose that a piece $P$ intersects $\freg_1$ is at least $2$ points, say, $a$ and $b$. We have to show that $P\subset \freg_1$. There are 3 possibilities.

(1a) $a,b\in \freg_{\max}$. Then $P\subseteq \freg_{\max}\subset \freg_1$ as $\freg_{\max}$ is a $\mathcal Q$-subtree.

(1b) $a\in \freg_{\max}$, $b\notin \freg_{\max}$. Let $c_1, c_2$ be arcs connecting $a$ to $z$ and $z$ to $b$, respectively. Then $c_1\cap c_2=z$ as $c_1\setminus \{ z\}$ and $c_2\setminus \{ z\} $ belong to different connected components of $\freg\setminus \{ z\} $. Hence the concatenation $c=c_1c_2$ is an arc. By Lemma \ref{DS} (a) we have $c\subseteq P$. In particular, $z\in P$ and consequently $P\in  \freg_1$.

(1c) $a,b\notin \freg_{\max}$. If $a$ and $b$ are from different connected components of $\freg_{\max}\setminus\{z\}$, we obtain $P\in \freg_1$ arguing as in (1b). Otherwise it is easy to verify that $P\in \freg_1$ by our construction of $\freg_1$.

Finally we observe that $\freg_1$ and $\free_1$ are good. Indeed it is easy to see that types of points in $\freg_{\max}\setminus \{z\} $ do not change, $z$ becomes of type I, and all other points of $\freg_1$ are of type II. The same holds for $\free_1$.  This  contradicts maximality of the triple $(\freg_{\max}, \free_{\max}, f_{\max})$.

\smallskip

{\it Case 2.} $[y,x]\cap \freg_{\max}=[y,z[$. Let
$\freg_1=\freg_{\max}\cup \{ z\}$. Since $\free$ is complete,
$f_{\max}$ extends to an isometry $\freg_1\to
\free_1=\free_{\max}\cup \{ w\} $ for some $w\in \partial
\free_{\max}$.  Clearly this isometry preserves types of points
since both $z$ and $w$ have type II and $\freg_1$ and $\free_1$ are
good $\mathcal Q$-subtrees of $\freg$ and $\free$ respectively by
Lemma \ref{comp}. This contradicts maximality of the triple
$(\freg_{\max}, \free_{\max}, f_{\max})$ again.

Thus $\freg_{\max}=\freg$, i.e., $\freg $ isometrically embeds in $\free$.

(3) Suppose now that $\freg$ and $\free $ are two universal $\qq
$-trees. Let the triple $(\freg_{\max}, \free_{\max}, f_{\max})$ be
defined as above. We already know that $\freg_{\max}=\freg$. Suppose
that there exists $x\in \free\setminus \free_{\max}$. Let $y$ be any
point in $\free_{\max}$. Then $[y,x]\cap \free_{\max}=[y,z]$ as
$\free_{\max} $ is complete being the isometric image of a complete
space. Let $z_0$ be the preimage of $z$ in $\freg_{\max}$. Clearly
$z_0$ is of type I while $z$ is of type II. This contradicts our
assumption that $f_{\max}$ preserves types of points. Thus
$\free=\free_{\max}$, i.e., $\freg$ and $\free$ are isometric.
\end{proof}

%%%%%%%%%%%%%%%%%%%%%%%%%%%%%%%%%%%%%%%%%%%%%%%%%%%%%%%%%%%%%%%%%%%%%%%%%%%%%%%%%%%%%%%%%%%%%%%%%%

\section{Asymptotic cones of relatively hyperbolic groups}

%%%%%%%%%%%%%%%%%%%%%%%%%%%%%%%%%%%%%%%%%%%%%%%%%%%%%%%%%%%%%%%%%%%%%%%%%%%%%%%%%%%%%%%%%%%%%%%%%%%

Throughout this section let $G$ denote a groups generated by a finite symmetric set $X$ and hyperbolic relative to a collection of proper subgroups $\{ H_1, \ldots , H_m\}$. Recall that the subgroups $H_1, \ldots , H_m$ are finitely generated in this situation. We assume that for every $i=1,\ldots , m$, $X$ contains a (finite) generating set $Y_i$ of $H_i$.

Given a non-principal ultrafilter $\omega $, and a scaling sequence $d$, let
$$
G(\omega, d)= \{ (g_i)^\omega \in \Pi ^\omega G \mid |g_i|=o_\omega (d_i)\} .
$$
Clearly $G(\omega, d)$ is a subgroup of $G^\omega $ and the induced action of $G(\omega, d)$ on $\CG $ fixes the point $e=(1,1,\ldots )^\omega \in \CG $. It follows from the Definition \ref{rhd} that $G(\omega, d)$ preserves the set of pieces of $\CG $, moreover a piece corresponding to $H_i$ is mapped to a piece of the same kind. In particular, $\CG $ preserves types of geodesics in $\CG$.  Hence $G(\omega, d)$ acts on the set $\mathcal {LC}_{\tau}(e)$ of limit connected components of $\CG\setminus \{ e\} $ of type $\tau $ as well as on the set $\mathcal P_i(e)$ of pieces of $\CG $ containing $e$ and corresponding to $H_i$. We say that an element $g\in G(\omega, d)$ is {\it hyperbolic} if $g=(g_i)^\omega $, where every $g_i$ is hyperbolic in $G$.

\begin{lem}\label{action}
Let $\omega $ a non-principal ultrafilter, $d$ a scaling sequence. Suppose that none of the peripheral subgroups is finite. Then no hyperbolic element of $G(\omega, d)$ preserves an element of $\mathcal P_i(e)$ or $\mathcal {LC}_\tau (e)$ for a non-trivial $\tau$.
\end{lem}

\begin{proof}Suppose that an element $x\in G(\omega, d)$ fixes a piece $\lio (g_iH_k)$, where $k\in \{ 1, \ldots , m\}$. Then by \cite[Lemma 4.19]{DS1}, $x$ is not hyperbolic.

Further let $C\in \mathcal {LC}_{\tau}(e)$ where $\tau$ is non-trivial. Suppose that a hyperbolic element $x\in G(\omega, d)$ fixes $C$. Let $\gamma $ be any (nontrivial) geodesic segment in $C\cup\{ e\}$ originated at $e$. Since $x$ is an isometry of $\CG $, it maps $\gamma $ to another geodesic segment $\delta $ in $C\cup\{ e\}$. By Lemma \ref{wd}, $\gamma$ and $\delta $ have at least one common point $u\ne e$. Clearly $x$ fixes $e$. Lemma \ref{DS} (c) now implies that $x$ fixes all pieces which intersect nontrivially the initial subsegment $[e,u]$ of $\gamma $ (or $\delta $). That set of pieces is non-empty because $\tau$ is non-trivial. However this contradicts the assumption that $x$ is hyperbolic.
\end{proof}

Given a word $W=h_1\ldots h_k$, where $h_j\in \mathcal H$ for $j=1, \ldots , k$,  we say that a path $p$ in $\Cay (G,X)$ is {\it corresponding} to $W$ if the following hold.
\begin{enumerate}
\item The path $p$ decomposes as $p=q_1\ldots q_k$.
\item If $h_j\in H_i$, then $\Lab (q_j)$ is a shortest word in the alphabet $Y_i$ representing $h_j$, $j=1, \ldots, k$.
\end{enumerate}

\begin{lem}\label{o}
There exist constants $\lambda$, $c$, $D$ with the following property. Let $W=h_1\cdots h_k$ be a word as above, where $|h_i|_X\ge D$. Suppose that no two consecutive letters $h_i, h_{i+1}$ belong to the same $H_i$. Then every path $p$ in $\Cay (G,X)$ corresponding to $W$ is $(\lambda , c)$-quasi-geodesic.
\end{lem}

\begin{proof}
Let $L$ be the constant from Lemma \ref{Omega}. Assume that
\begin{equation}\label{eD}
D=L+2
\end{equation}
(the exact value of $D$ will be specified later.) Let $p^\prime $ be the path in $\Cay (G, X\cup\mathcal H)$ obtained from $p$ by replacing every component of $p$ with a single edge, i.e., every $q_j$ is replaced by an edge $r_j$ labeled by $h_j$. Thus $p^\prime =r_1\ldots r_k$. Let us show that all components of $p^\prime $ are isolated.

Indeed suppose that some components, say $r_i$ and $r_{i+a}$, $a\ge 1$, of $p^\prime $ are connected.  By the assumption of the lemma, $a\ge 2$. Let $e$ be an edge in $\Cay (G, X\cup\mathcal H)$ connecting $(r_i)_+$ to $(r_{i+a})_-$. Consider the cycle $c=r_{i+1}\ldots r_{i+a-1}e^{-1}$. Without loss of generality we can assume that $a$ is minimal possible and hence all components of $c$ are isolated. The total $X$-length of all components of $c$ is at most $Ll(c)=L(a-1)+1$ by Lemma \ref{Omega}. On the other hand, it is at least $D(a-1)>(L+1)(a-1) \ge L(a-1)+1$ by assumptions of our lemma and (\ref{eD}). This contradiction implies that all components of $p^\prime $ are isolated.

Let now $v$ be a subpath of $p$ such that all components of $v$ have $X$-length at least $D$. We denote by $r$ the number of components in $v$. Let $w$ be a geodesic in $\cgx$ connecting $v_+$ to $v_-$. Since $\Lab (w)$ is a word in $X$, $w$ has no components at all.  Thus all components of the cycle $vw$ in $\Cay (G, X\cup\mathcal H)$ are isolated. By Lemma \ref{Omega} we have $l(v)\le L(l(w)+r)$, which yields
\begin{equation}\label{lv}
\dx (v_-, v_+)=l(w)\ge l(v)/L -r \ge (l(v)-r)/L\ge (D-1)l(v)/(DL).
\end{equation}
Now let $u$ be an arbitrary subpath of $p$. Note that every component of $u$, except possibly for the first and the last ones, has $X$-length at least $D$. Removing components of length less than $D$ from $u$ if necessary, we obtain a path $v$ which satisfies (\ref{lv}). Hence  $\dx (u_-, u_+)\ge (D-1)l(u)/(DL) - 2D$.
\end{proof}

\begin{lem}\label{ttr} The limit $\lio \la E(f_1)\ra $ in $\CG$ is a bi-infinite geodesic inside the transversal tree of $\CG$.
\end{lem}

\proof It follows from the fact that, by the definition of  relatively hyperbolic groups, $\CG$ is tree-graded with respect to limits of cosets of $H_i$ and limits of cosets of $\la E(f_i)\ra$, $i=1,2$.\endproof

We say that two collections of metric spaces $\mathcal Q=\{ Q_i\}_{i\in I}$ and $\mathcal R=\{ R_i\} _{i\in I}$ are {\it uniformly bi-Lipschitz equivalent} if there is a family of bi-Lipschitz maps $\phi_i\colon Q_i\to R_i$ with uniformly bounded constants. The lemma below follows immediately from the definition of the metric in tree products.

\begin{lem}\label{ble}
Suppose that two collections $\mathcal Q=\{ Q_i\}_{i\in I}$ and $\mathcal R=\{ R_i\} _{i\in I}$ of homogeneous complete geodesic spaces are uniformly bi-Lipschitz equivalent. Then the tree products $\Pi\qq$ and $\Pi\rrr$ are bi-Lipschitz equivalent.
\end{lem}

\begin{thm}\label{rhg} Let $G$ be a group generated by a finite set $X$ and
hyperbolic relative to a collection of subgroups $\{ H_1, \ldots ,
H_n\} $. Then for every non-principal ultrafilter $\omega $ and
every scaling sequence $d=(d_i)$, the asymptotic cone $\CG $ is
bi-Lipschitz equivalent to the universal $\mathcal Q$-tree, where
$\mathcal Q=\{ \Con ^\omega (H_i, d)\mid i=1, \ldots , n\}$.
\footnote{Although it is  true that a metric space which is
bi-Lipschitz equivalent to a $\mathcal Q$-tree is itself isometric
to a $\mathcal Q'$-tree with pieces from $\mathcal Q'$ bi-Lipschitz
equivalent to pieces from $\mathcal Q$, the size of $\mathcal Q'$
may not be the same as the size of $\mathcal Q$, and the $\mathcal
Q'$-tree may not be universal even if the $\mathcal Q$-tree is
universal.}
\end{thm}

\begin{proof}

Clearly it suffices to prove the theorem for any particular finite
generating set $X$ of $G$. In what follows we assume (as above) that
$X$ contains a generating set $Y_i$ of $H_i$ for every $i=1, \ldots
, n$. In addition we assume that $X$ also contains two hyperbolic
elements $g_1,g_2\in G$ such that $G$ is hyperbolic relative to $\{
H_1, \ldots , H_n, E(g_1)\}$ as well as relative to $\{ H_1, \ldots
, H_n, E(g_1), E(g_2)\}$.

Recall that by Definition \ref{rhd}, $\CG$ is isometric to a
$\mathcal Q_0$-tree, where $\mathcal Q_0=\{ \Con ^\omega ((H_i, \d
_X \vert_{H_i}), d) \mid i=1, \ldots , n\}$. Here $\Con ^\omega
((H_i, \d _X \vert_{H_i}), d)$ is the asymptotic cone of $H_i$
endowed with metric induced from $G$. Recall that subgroups $H_i$
are undistorted in $G$. This is well known and follows, for example,
from Lemma \ref{o} applied in the case $k=1$. Hence $\Con ^\omega
((H_i, \d _X \vert_{H_i}), d)$ is bi-Lipschitz equivalent to
asymptotic cones $\Con ^\omega (H_i, d)$, where $H_i$ is endowed
with an intrinsic word metric (e.g., with respect to $Y_i$). Thus we
only need to show that the $\qq _0$-tree $\CG $ is universal. Then
applying Lemma \ref{ble} finishes the proof.

Since every asymptotic cone of a group is homogeneous, geodesic, and
complete, we only have to verify conditions (b) and (c) in
Definition \ref{defu}. Without loss of generality we can assume that
$s$ is represented by $(1,1, \ldots )^\omega $.

We start with (c). Since every $P$ is homogeneous (being an
asymptotic cone of $H_i$), the choice of $x$ does not matter, and
instead of $(Q,x)$-pieces containing $s$, we can consider $Q$-pieces
containing $s$ (i.e. pieces containing $s$ that are isometric to
$Q$). Note that for every sequence of numbers $n=(n_j)$ with
$n_j=o_\omega(d_j)$, the element of $G(\omega, d)$ represented by
$(g_1^{n_j})^\omega$ is hyperbolic. Clearly there are uncountably
many pairwise distinct such elements. Hence condition (c) holds by
Lemma \ref{action}.

Now let $\theta\in \Theta$ an equivalence class of a $\qq _0$-type.
Suppose first that $\theta$ is trivial, i.e., it is the equivalence
class of the type $(\emptyset, \emptyset)$. Let $\qq^\prime
=\qq_0\cup\{\lio E(g_1)\} $.  It follows from Definition \ref{rhd}
that we can also think of $\CG $ as a $\qq^\prime $-tree. Applying
Lemma \ref{action} to the piece $E=\lio E(g_1)$, we obtain a set of
cardinality continuum of isometric copies $\{E_\alpha \}_{\alpha \in
A}$ of $E$ such that $E_\alpha \cup E_\beta =\{ s\}  $ whenever
$\alpha \ne \beta $. Let $C_\alpha $ be a connected component of
$\CG \setminus \{ s\} $ which intersects $E_\alpha $ nontrivially.
By definition \ref{defu} applied to the collection of peripheral
subgroups $\{ H_1, \ldots , H_n, E(g_1)\}$, every $E_\alpha $
intersects every piece of the $\qq _0$-tree $\CG $ in at most one
point. In particular, $C_\alpha $ belongs to the transversal tree of
the $\qq _0 $-tree $\CG $ at $s$. By Lemma \ref{ctree}, every
$C_\alpha $ has type $\theta$. This proves (c) for trivial $\theta
$.

Now suppose that $\theta$ is not trivial. We fix a representative
$\tau=(U,f)$ of the equivalence class $\theta$. Without loss of
generality we can assume that $U\subset [0,a]$ for some $a\in
\mathbb R_+$. Let us enumerate intervals in $U$ in an arbitrary
order, and for every $m=1,2,...$, let $U_m$ be the subset of $U$
consisting of the union of the first $m$ intervals. We consider the
partition of the interval $[0,a]$ whose classes are intervals from
$U_m$ and the intervals of the complement $[0,a]\setminus U_m$.
Remove the endpoints of all intervals of that partition. The
resulting collection of open intervals will be denoted by $V_m$.
Clearly, $U_m\subseteq V_m$. We number intervals in $V_m$ from left
to right: $V_m^1,...,V_m^{s_m}$, where $s_m=\# V_m$.

The map $f$ associates to each interval $]p,q[$ in $U$ an element of $(u^{]p,q[},w^{]p,q[})\in D(Q(p,q))$ for some $Q(p,q)=\Con ^\omega ((H_{k(p,q)}, \d _X \vert_{H_{k(p,q)}}), d)\in \qq_0$. Since $Q(p,q)$ is homogeneous, we can assume that $u^{(p,q)}$ is the point of $\Con ^\omega ((H_{k(p,q)}, \d _X \vert_{H_{k(p,q)}}), d)$ represented by $(1)^\omega$. Suppose that $w^{(p,q)}$ is represented by $(w^{({p,q})}_i)^\omega$, where $w^{({p,q})}_i$ are elements of $H_{k(p,q)}$. In particular, we have
\begin{equation}\label{lwpq}
\lio |w^{({p,q})}_i|_X/d_i=q-p.
\end{equation}
We represent $w^{({p,q})}_m$ by a shortest word $W^{({p,q})}_m$ in the generators $Y_{k(p,q)}$.

Below we say ``$X$-length" instead of ``length with respect to the generating set $X$" for brevity. For each $m$ and each $(p,q)\in V_m$, we construct a word $Z^{({p,q})}_m$ as follows. If $(p,q)\in U_m$, we set $Z^{({p,q})}_m= T^{({p,q})}_mS^{({p,q})}_m $, where $S^{({p,q})}_m$ is a prefix of the word $W^{({p,q})}_m$ and $T^{({p,q})}_m$ is a power of $g_1$ such that
\begin{itemize}
\item[(Z$_1$)] The $X$-length of the element represented by the word $Z^{({p,q})}_m=T^{({p,q})}_mS^{({p,q})}_m$ is equal to $|w^{({p,q})}_m|_X+c^{({p,q})}_m$ where $|c^{({p,q})}_m|=O(1)$ $\omega$-a.s.
\item[(Z$_2$)] The $X$-length of the element represented by $T^{({p,q})}_m$ is at least $D$, where $D$ is the constant in Lemma \ref{o}, $\omega$-a.s.
\end{itemize}
If $(p,q)\notin U_m$, then $Z^{({p,q})}_m$ is a power of $g_2$ such that
\begin{itemize}
\item[(Z$_3$)] The $X$-length of $Z^{({p,q})}_m$ is $(q-p)d_m+c^{({p,q})}_m$ where $|c^{({p,q})}_m|=O(1)$.
\end{itemize}
Moreover, it is not hard to see that the (not necessarily positive) constants $c^{({p,q})}_m$ and words $T^{({p,q})}_m$ can be chosen in such a way that for some constant $r>0$ we have:
\begin{itemize}
\item [(Z$_4$)] For every $m$, and every $(p,q)\in V_m$, the $X$-length of the element represented by $T^{({p,q})}_m$ is at most $r$.
\item [(Z$_5$)] For every $k=1,...,|V_m|$, $\left| \sum_{j=1}^k c^{(V_m^j)}_m\right|\le r$.
\end{itemize}

Let us consider the concatenation of all words $Z^{({p,q})}_m$, ${(p,q)}\in V_m$, counting the subintervals of $[0,a]$ from $V_m$ from left to right. The set of words obtained in this way will be denoted by ${\mathcal Z}_m$.  Denote by $\g_m$ the path in $\Cay (G,X)$ starting at $1$ and labeled by a word $Z_m\in{\mathcal Z}_m$. We parameterize it by length $\g_m\colon [0, l_m]\to \Cay (G,X)$. Observe that $\lio (l_m/d_m)=a$ by (\ref{lwpq}), (Z$_1$)--(Z$_3$), and (Z$_5$).

Let us define a limit path $\g \colon [0,a]\to \CG $ by the rule $$\g (t)=\lio \g_m (\min\{ td_m, l_m\})$$ for $t\in [0,a]$. It is easy to see that $\g $ is indeed a path, i.e., the map is continuous. Note also that by Lemma \ref{o} applied to the collection of peripheral subgroups $\{ H_1, \ldots , H_n, E(g_1), E(g_2)\}$ and by properties (Z$_1$)--(Z$_3$), $\g_m$ is $(\lambda,c)$-quasi-geodesic for every $m$, where the constants $\lambda , c$ are independent of $m$. This implies that $\g$ is a $\lambda$-bi-Lipschitz map, thus it defines an arc in $\CG$.
We want to show that, in fact, $\g $ is geodesic and its type is $\tau $.

Note that if for a simple path $\p$ is a tree-graded space, all nontrivial intersections with transversal trees and pieces are geodesic, then $\p$ is geodesic. Indeed let $\q$ be a geodesic connecting the endpoints of $\p$. Let $p_1, p_2, \ldots $ (respectively $q_1,q_2,\ldots $) be the list of all nontrivial intersections of $\p $ (respectively, $\q$) with pieces or  transversal trees. By part (c) of Lemma \ref{DS}, we have $p_i$ has the same endpoints as $q_i$ (up to changing the enumeration). Note that $l(\p )=\sum\limits_{i} l(p_i)$ and $l(\q )=\sum\limits_{i} l(q_i)$ as $\p \setminus (p_1\cup p_2\cup \ldots )$ has measure $0$ and similarly for $\q$. As every $p_i$ is geodesic, $l(p_i)=l(q_i)$. Therefore the above sums are equal, i.e. $l (\p )=l(\q)$. This implies that $p$ is geodesic.

It is straightforward to verify using (Z$_1$)--(Z$_5$), (\ref{lwpq}), and the definition of $\g $ that if $]p,q[\in U$ then the subpath $\g[p,q]$ is a geodesic inside a piece isometric to $Q(p,q)$ and there is an isometry that takes this piece to $Q(p,q)$ while taking $\g(p)$ to the point represented $(1)^\omega$, and $\g(q)$ to $w_{]p,q[}$. Similarly if an interval $]p,q[$ is a connected component of the complement $[0,a]\setminus U$, then $\g]p,q[$ is a geodesic inside a transversal tree of $\CG$ by Lemma \ref{ttr}. In particular, it intersects all pieces of $\CG$ trivially. Thus the only nontrivial (i.e., consisting of more than one point) intersections of $\g $ with pieces of $\CG$ are segments $\g[p,q]$.

The previous two paragraphs imply that the path $\g $ is geodesic and its type is $\tau $. Hence $\CG\setminus \{s\} $ has a limit connected component of type $\theta $. It remains to note that by Lemma \ref{action}, the cardinality  of the set of connected components of type $\theta$ is continuum.
\end{proof}

Corollaries \ref{cor1} and \ref{cor2} follow from the theorem and Lemma \ref{ble} immediately.

\section{Tree-graded asymptotic cones and the Continuum Hypothesis}

We start by recalling some ideas from \cite{KSTT}. Although the authors of \cite{KSTT} only deal with metric spaces arising from finitely generated groups, most of the theory generalizes to the general case without any changes.

Consider the first order language $\mathcal L$ (with equality) consisting of a constant symbol $e$ and countably many binary predicates $\{ R_r \}_{r\in \mathbb Q^+} $ indexed by positive rational numbers. We say that an $\mathcal L$-structure $A$ is a {\it KSTT structure} if $A$ satisfies the following axioms.
\begin{enumerate}
\item[(A$_1$)] $\forall \, r\in \mathbb Q^+$ $\forall \, x\in A$ $R_r(x,x)$.
\item[(A$_2$)] $\forall \, r\in \mathbb Q^+$ ($R_r(x,y)$ $\Rightarrow $ $R_r(y,x)$).
\item[(A$_3$)] $\forall $ $r,s\in \mathbb Q^+$ such that $r<s$ ($R_r(x,y)$ $\Rightarrow$ $R_s(x,y)$).
\item[(A$_4$)] $\forall $ $r,s\in \mathbb Q^+$ ($R_r(x,y)$ $\& $ $R_s(y,z)$ $\Rightarrow$ $R_{r+s}(x,z)$).
\end{enumerate}

To each KSTT structure $A$, one can canonically associate a pointed metric space $\mu (A)$ as follows. Let $$A_b=\{x\in A\mid \exists r>0\, A\models R_r(x,e)\}.$$ Let $\approx$ be the relation on $A_b$ defined by $$x\approx y\;\;\; \Leftrightarrow\;\;\; \forall r>0\, A\models R_r(a,b).$$ It is easy to check that $\approx $ is an equivalence relation. Let $\mu (A)=A_b/\approx $. We define a distance function on $\mu (A)$ by the rule $$\dist([x],[y])=\inf\{r\mid A\models R_r(x,y)\}.$$  Axioms (A$_1$)--(A$_4$) guarantee that $\dist$ is well-defined (i.e., is independent of the choice of representatives of the equivalence classes $[x]$ and $[y]$) and satisfies all properties of a metric.

Conversely, every metric space $M$ with a basepoint can be considered as a KSTT structure in a natural way. The universe of the structure is $M$,  $M\models R_r(x,y)$ iff $\dist (x,y)\le r$, and $e$ interprets as the fixed point.

The following two lemmas are essentially contained in \cite{KSTT}. The proofs are straightforward and we leave them to the reader. By $\cong $ we denote the isometry relation between metric spaces.

\begin{lem}\label{am1}
If two KSTT structures $A_1, A_2$ are isomorphic, then $\mu (A_1)\cong \mu (A_2)$.
\end{lem}

\begin{lem}\label{am2}
Let $S$ be a metric space, $\omega $ an ultrafilter, $d=(d_i)$ a scaling sequence, $o=(o_i)$ a sequence of observation points. Then $\Con ^\omega (S, d,o)=\mu (\prod ^\omega S_i)$, where $S_i$ is the metric space $(S, \frac1{d_i} \dist )$ with fixed point $o_i$ considered as a KSTT structure.
\end{lem}

Let $M$ be a metric space, $\omega$ an ultrafilter, and $(o_i)$ a sequence of points $o_i\in M$. Recall that the $\omega$-limit of $M$ with respect to $(o_i)$, denoted by $\lio (M, o_i)$, is defined in the same way as the asymptotic cone of $M$ but with all scaling constants equal to $1$. If all $o_i$ are the same, i.e. $o=o_i, i=1,2,...$, the $\omega$-limit does not depend (as a metric space without basepoint) on the choice of $o$, and we denote it by $\lio(M)$.

It is a standard fact (see \cite[Corollary 3.24]{DS}) that any $\omega$-limit of an asymptotic cone of a metric space is again an asymptotic cone of that space. The next lemma shows that, assuming Continuum Hypothesis, we do not get new asymptotic cones this way. (This lemma was presented as a remark in \cite{DMS} without a proof.)

\begin{lem}[Assuming CH is true]\label{ch}
Let $C=\Con ^\rho (S,(d_i),(o_i))$ be an asymptotic cone of a metric space $S$ of cardinality at most continuum. Then for every ultrafilter $\omega $, the $\omega$-limit $\lio C$ is isometric to $C$.
\end{lem}

\proof
Let us denote by $A$ the ultraproduct $\prod^\rho S_i$ of the metric spaces $S_i=(S, \frac1{d_i} \dist )$ with fixed points $o_i\in S_i$ considered as KSTT structures. Recall that $e$ interprets as $o=(o_i)^\rho $ in $A$. To each $x=(x_i)^\omega \in \prod^\omega A$, we associate the element $\bar x =([x_i])^\omega \in \prod^\omega \mu (A)$.  Observe that for every $r\in \mathbb Q^+$, we have
$$
\begin{array}{rl}
A \models R_r(x,e)\;\Leftrightarrow \;& \omega (\{ i\in \mathbb N\mid A\models R_r(x_i, o_i)  \} )=1 \;\Leftrightarrow\;\\&\\
& \omega (\{ i\in \mathbb N\mid \dist _{\mu (A)} (x_i, o_i)\le r  \} )=1 \;\Leftrightarrow\;\\&\\
& \omega (\{ i\in \mathbb N\mid \mu (A) \models R_r([x_i], [o_i])  \} )=1 \;\Leftrightarrow\;\\&\\
& \prod^\omega \mu (A) \models R_r (\bar x, e)
\end{array}
$$
This implies that the map $x\mapsto \bar x$ that takes $\left(\prod ^\omega A\right)_b$ to $\left(\prod ^\omega \mu (A)\right)_b$, preserves the equivalence relation $\approx $, and induces an isometry $\mu \left(\prod^\omega A\right)\to \mu \left(\prod^\omega \mu(A)\right)$. Furthermore, by Lemma \ref{am2}  $\mu \left(\prod ^\omega \mu(A)\right)=\lio C$.  Hence
\begin{equation}\label{ch1}
\mu \left(\prod ^\omega A\right)\cong\lio C.
\end{equation}

Finally we note that every ultraproduct with respect to a
non-principal ultrafilter is $\aleph_1$-saturated \cite{CK} and has
cardinality $2^{\aleph_0}$. Thus if CH holds it is saturated.
Therefore $\prod ^\omega A$ and $A=\prod ^\rho S_i$ are both
saturated (since the cardinality of $A$ is at most continuum). By
the \L o$\acute{s}$ Theorem $A$ and $\prod ^\omega A $ are
elementary equivalent. Recall that every two elementary equivalent
saturated models having the same cardinality are isomorphic
\cite{CK}. Thus $\prod ^\omega A$ is isomorphic to $A$. Now Lemma
\ref{am1} implies that $\mu \left(\prod ^\omega A\right) \cong \mu
(A)$. Comparing this to (\ref{ch1}) and applying Lemma \ref{am2}
again, we conclude that $\lio C\cong \mu (A)=C$.
\endproof

\begin{rem}
Note that even if CH holds, the natural isometric embedding $C\to
\lim^\omega(C)$ is not surjective in general.
\end{rem}

The next lemma follows from \cite[Proposition 3.20]{DMS} and  \cite[Theorem 3.30]{DS}.

\begin{lem}\label{dms} Let $S$ be a homogeneous tree-graded metric space. Then every ultralimit $\lio S$ has a tree-graded structure with a non-trivial transversal tree at every point. Moreover every piece of $\lio S$ is naturally isometric to $\lio (P_i,o_i)$ for some sequence $(P_i,o_i)$, where $P_i$ are pieces of $S$ and $o_i$ are observation points.
\end{lem}

\begin{thm} [Assuming CH is true]\label{ttt5}
Let $\free $ be an asymptotic cone of a geodesic metric space of cardinality at most continuum.
Suppose that $\free $ is homogeneous and has cut points. Then $\free $ is a universal $\qq$-tree, where $\qq$ consists of representatives of isometry classes of maximal connected subspaces without cut points in $\free $.
\end{thm}

\begin{proof}

Recall that every space with cut points is tree-graded with respect
to maximal connected subsets without cut points. Hence $\free $ is a
$\qq$-tree.  As every asymptotic cone is complete, we only have to
check conditions (b) and (c) in Definition \ref{defu}. Since we
assume CH, by Lemma \ref{ch}, $\free$ is isometric to $\free_1=\lio
\free $. Hence it suffices to verify these conditions for $\free
_1$.

Let us start with (c). We fix some $x\in \free _1$. Without loss of
generality we can assume that $x$ is the $\omega $-limit of a
constant sequence $(o,o, \ldots )$ for some $o\in \free $. Let $Q$
be a piece of $\free _1$ containing $x$. By Lemma \ref{dms},
$(Q,x)=\lio (Q_i,o_i)$ for some pieces $Q_i$ of $\free $ and
observation points $o_i\in Q_i$. Again without loss of generality we
can assume that $o_i=o $ for all $i$.

Let $(a_i)$, $(b_i)$ be sequences of elements of $\free $ such that
$\lio a_i=\lio b_i=x$ and $a_i,b_i\in T_o$. Let $\alpha _i$ and
$\beta _i$ be some isometries of $\free $ such that $\alpha_i
(o)=a_i$, $\beta _i(o)=b_i$. Finally let $(A,x)=\lio (\alpha
_i(Q_i), a_i)$ and $(B,x)=\lio (\beta _i (Q_i), b_i)$. Clearly
$(A,x)$ and $(B,x)$ are pieces isometric to $(Q,x)$. Observe that if
$(a_i)^\omega \ne (b_i)^\omega $, then $A\ne B$. Indeed let $p=\lio
p_i$ be a point in $A$ such that $d(p, x)>\e>0$. Then $d(p_i,
\beta_i (Q_i))=d(p_i,b_i)>\e $ \oas. Hence $A\ne B$. Since $T_o$ is
nontrivial by Lemmas \ref{dms} and \ref{ch}, this implies part (c)
of Definition \ref{defu}.

Arguing as in the proof of Theorem \ref{rhg}, one can show that
every $\theta\in \Theta $ is realized as the type of a limit
connected component at some point of $\free _1$. One has to use
paths in pieces of $\free $ (respectively, in transversal trees)
instead of words in the alphabets $Y_i$ (respectively, powers of
$f_1$ and $f_2$). Finally, the same proof as for pieces, shows that
there are uncountably many connected components of type $\theta $ at
every point.
\end{proof}

Note that even without CH the proof of Theorem \ref{ttt5} together
with \cite[Corollary 3.24]{DS} still implies that {\em some}
asymptotic cone (probably with respect to a different ultrafilter)
is a universal $\qq$-tree.


\begin{thebibliography}{99}
\bibitem{Beh}
J.A. Behrstock,  Asymptotic geometry of the mapping class group and
  {T}eichm\"{u}ller space ,  Geom. Topol. 10  (2006), 1523--1578.

\bibitem{BC} J. Behrstock, R. Charney,  Divergence and quasimorphisms of right-angled Artin groups, arXiv:1001.3587.

\bibitem{BKMM} J. Behrstock, B. Kleiner, Y. Minsky, L. Mosher, Geometry and rigidity of mapping class groups, arXiv:0801.2006.

\bibitem{BDS}J. Behrstock, C. Dru\c tu, M. Sapir,  Median structures on asymptotic cones and homomorphisms into mapping class groups, arXiv:0810.5376, accepted in Proc. London Math. Soc., 2010.

\bibitem{BDS1} J. Behrstock, C. Dru\c tu, M. Sapir, Homomorphisms into mapping class groups. An addendum.    arXiv:1005.5056, accepted in Proc. London Math. Soc, 2010.

\bibitem{Br} E. Breuillard,
Geometry of groups of polynomial growth and shape of large balls, arXiv 0704.0095.

\bibitem{CK} C. C. Chang and H. J. Keisler, Model Theory, North-Holland, Amsterdam, 1973.



\bibitem{DS}C. Dru\c tu,  M. Sapir. Tree-graded spaces and asymptotic
cones of groups. With an appendix by D. Osin and M. Sapir.
  Topology  44  (2005),  no. 5, 959--1058.

\bibitem{DS1} C. Dru\c tu, M.Sapir. Groups acting on tree-graded spaces and splittings of relatively
hyperbolic groups,   Adv. Math.  217  (2008), no. 3,
1313--1367.

\bibitem{DMS} C. Dru\c tu, S. Mozes, M. Sapir, Divergence of lattices in semi-simple Lie groups and graphs of groups,  arXiv:0801.4141.

\bibitem{deC} Y. de Cornulier, Dimension of asymptotic cones of Lie groups.  J. Topol.  1  (2008),  no. 2, 342-361.

\bibitem{EP} A. Dyubina (Erschler), I. Polterovich,
Explicit constructions of universal $\mathbb R$-trees and asymptotic geometry of hyperbolic spaces.
  Bull. London Math. Soc.  33  (2001), no. 6, 727--734.

\bibitem{F}
B. Farb, Relatively hyperbolic groups, Geom. Funct. Anal. 8 (1998),
810--840.

\bibitem{Gro}
M. Gromov, Hyperbolic groups, in: Essays in group theory (ed. S M Gersten),. MSRI Publications 8, Springer-Verlag, 1987,  75-265.

\bibitem{KSTT}
L.~Kramer, S.~Shelah, K.~Tent, and S.~Thomas, Asymptotic cones of
  finitely presented groups, Adv. Math. 193 (2005), no. ~1, 142--173.


\bibitem{MNO} J. Mayer, J. Nikiel and L. Oversteegen, `Universal spaces for R-trees', Trans. Amer. Math. Soc. 334 (1992) 411-432.

\bibitem{OOS} A. Olshanskii, D.  Osin, Mark  Sapir,  Lacunary hyperbolic groups. With an appendix by M. Kapovich and B.  Kleiner.  Geom. Topol.  13  (2009),  no. 4, 2051--2140.

\bibitem{OlshSap} A. Olshanskii, M.  Sapir, A finitely presented group with two non-homeomorphic asymptotic cones.  Internat. J. Algebra Comput.  17  (2007),  no. 2, 421--426.

\bibitem{Osi06a} D. Osin, Relatively hyperbolic groups: intrinsic geometry, algebraic properties, and algorithmic problems,
 Mem. Amer. Math. Soc. 179  (2006),  no. 843.

\bibitem{Osi06b}
D. Osin, Elementary subgroups of relatively hyperbolic groups and
bounded generation,   Internat. J. Algebra Comput.  16
(2006), no. 1, 99--118.

%\bibitem{Osi07} D.V. Osin, Peripheral fillings of relatively hyperbolic groups,
%{\it Invent. Math.}  {\bf 167}  (2007),  no.~2, 295--326.

\bibitem{Pansu} P. Pansu, Croissance des boules et des g\'eod\'esiques ferm\'ees dans les nilvari\'et\'es, Ergod.
Theory Dynam. Syst. 3 (1983), 415--445.

\bibitem{Q}
G. Quenell, Combinatorics of free product graphs, Contemp. Math. 173 (1994), 257-281.

\bibitem{Si1} Alessandro Sisto, Tree-Graded Spaces and Relatively Hyperbolic Groups,
http://etd.adm.unipi.it/theses/available/etd-06082010-125639/,  Tesi di laurea specialistica, Universit\`a de Pisa,
June, 2010.


\bibitem{Si2} Alessandro Sisto, Bilipschitz types of tree-graded spaces, arXiv:1010.4552, 2010.



\bibitem{TV} Simon Thomas, Boban Velickovic, Asymptotic cones of finitely generated groups.  Bull. London Math. Soc.  32  (2000),  no. 2, 203--208.

\bibitem{Tuk}
P. Tukia, Convergence groups and Gromov's metric hyperbolic spaces,
New Zealand J. Math 23  (1994), 157--187.

\bibitem{Yam}
A. Yaman, A topological characterization of relatively hyperbolic
groups,   J. Reine Angew. Math.  566  (2004), 41--89.


\end{thebibliography}
\end{document}